\documentclass[10pt,a4paper]{article}

\usepackage{hyperref}
\usepackage[utf8]{inputenc}
\usepackage[T1]{fontenc}
\usepackage{amssymb}
\usepackage[english]{babel}
\usepackage{amsmath}
\usepackage{amsthm}
\usepackage{amsfonts}
\usepackage{tikz-cd}
\usepackage{color}
\usepackage{tikz}
\usepackage{enumerate}
\usepackage{enumitem}
\usepackage{changepage} 
\usepackage{listings}             
\usetikzlibrary{arrows}

\makeatletter
\newcommand{\subjclass}[2][2020]{%
  \let\@oldtitle\@title%
  \gdef\@title{\@oldtitle\footnotetext{#1 \emph{Mathematics subject classification.} #2}}%
}
\newcommand{\keywords}[1]{%
  \let\@@oldtitle\@title%
  \gdef\@title{\@@oldtitle\footnotetext{\emph{Key words and phrases:} #1.}}%
}
\makeatother

\date{9th March 2022}
\title{Distinguishing secant from cactus varieties}
\newtheorem{theorem}{Theorem}[section]
\newtheorem{corollary}[theorem]{Corollary}
\newtheorem{remark}[theorem]{Remark}
\newtheorem{lemma}[theorem]{Lemma}
\newtheorem{proposition}[theorem]{Proposition}
\theoremstyle{definition}
\newtheorem{definition}[theorem]{Definition}
\newtheorem{example}[theorem]{Example}
\DeclareMathOperator{\Spec}{Spec} \DeclareMathOperator{\Proj}{Proj} \DeclareMathOperator{\brr}{br} \DeclareMathOperator{\rr}{r} \DeclareMathOperator{\crr}{cr} \DeclareMathOperator{\bcrr}{bcr}

\DeclareMathOperator{\Ann}{Ann}
\DeclareMathOperator{\image}{im}
\DeclareMathOperator{\length}{length}
\DeclareMathOperator{\Slip}{Slip}
\newcommand{\Hilb}{\mathcal{H}{ilb}}
\DeclareMathOperator{\Gr}{Gr}
\DeclareMathOperator{\Apolar}{Apolar}

\DeclareMathOperator{\pr}{pr}
\DeclareMathOperator{\Hom}{Hom}

\DeclareMathOperator{\Sym}{Sym}
\DeclareMathOperator{\codim}{codim}

\def\AA{{\mathbb A}}
\def\CC{{\mathbb C}}

\def\PP{{\mathbb P}}

\def\ZZ{{\mathbb Z}}

\def\kk{{\Bbbk}}
\def\btd{{\blacktriangledown}}

\begin{document}

\author{
Maciej Ga\l{}\k{a}zka\thanks{ Faculty of Mathematics, Computer Science and Mechanics, University of Warsaw, ul.~Banacha 2, 02-097 Warsaw, Poland (\url{mgalazka@mimuw.edu.pl}) (corresponding author)} \and
Tomasz Ma\'{n}dziuk\thanks{ Faculty of Mathematics, Computer Science and Mechanics, University of Warsaw, ul.~Banacha 2, 02-097 Warsaw, Poland (\url{t.mandziuk@mimuw.edu.pl})} \and
Filip Rupniewski\thanks{ 
Institute of Mathematics of the Polish Academy of Sciences, ul.~\'Sniadeckich 8, 00-656 Warsaw, Poland (\url{f.rupniewski@impan.pl})}
}


\maketitle

\begin{abstract}
  Cactus varieties are a generalization of secant varieties. They are defined using linear spans of arbitrary finite schemes of bounded length, while secant varieties use only isolated reduced points. In particular, any secant variety is always contained in the respective cactus variety, and, except in a few initial cases, the inclusion is strict. It is known that lots of natural criteria that test membership in secant varieties are actually only tests for membership in cactus varieties. In this article, we propose the first techniques to distinguish actual secant variety from the cactus variety in the case of the Veronese variety. We focus on two initial cases, $\kappa_{14}(\nu_d(\PP^n))$ and $\kappa_{8,3}(\nu_d(\PP^n))$, the simplest that exhibit the difference between cactus and secant varieties.
We show that for $d\geq 5$, the component of the cactus variety $\kappa_{14}(\nu_d(\PP^6))$ other than the secant variety $\sigma_{14}(\nu_d(\PP^6))$
consists of degree $d$ polynomials divisible by a $(d-3)$-rd power of a linear form. We generalize this description to an arbitrary number of variables.
We present an algorithm for deciding whether a point in the cactus variety  $\kappa_{14}(\nu_d(\PP^n))$ belongs to the secant variety $\sigma_{14}(\nu_d(\PP^n))$ for $d\geq 6,$ $n \geq 6$.
We obtain similar results for the Grassmann cactus variety $\kappa_{8,3}(\nu_d(\PP^n))$.
Our intermediate results give also a partial answer to analogous problems for other cactus varieties and Grassmann cactus varieties to any Veronese variety.
\end{abstract}

\section{Introduction}\label{s:introduction}

The topic of secant varieties and ranks goes back to works of Sylvester on apolarity in the 19th century. But at that time, the problems of computing ranks of tensors or polynomials were inaccessible (except for a few initial cases). Nowadays, the situation is better---this is mainly due to the
development of algebraic geometry, representation theory, commutative algebra, and (last but not least) computational power, thanks to which substantial progress has been made.

Throughout this article, we consider the Waring rank of a homogeneous
polynomial $G$ of degree $d$, that is, the smallest number $r$ such that $G$ is
a sum of $r$ $d$-th powers of linear forms (such a sum is called a minimal
decomposition). By analogy with analyzing complicated data coming from the physical
world, the rank should correspond to the number of simple ingredients affecting
our complicated state. Waring decomposition is applicable in the process of blind identification of underdetermined mixtures. For more details and other applications see \cite{C02}, \cite{CGLM08}, \cite{CM96},  and the references therein.

There are two basic problems to solve: determining the rank, and finding an
explicit minimal decomposition. In this article we address the former problem.

In order to better understand the notion of rank, we consider the secant
varieties of the Veronese variety. Let us start with a complex vector space
$\mathbb{C}^{n+1}$.  For a positive integer $d$ consider the map of projective spaces $\nu_d \colon
\mathbb{P} \CC^{n+1} \to \mathbb{P}(\Sym^d \CC^{n+1})$, which
assigns to a form its $d$-th power. 
Here $\mathbb{P}$ denotes the naive
projectivization of a vector space, i.e.\ its set of lines through $0$ and $\Sym^d \CC^{n+1}$ denotes the space of symmetric tensors of order $d$ in $n+1$ variables. The $r$-th secant
variety $\sigma_r(\nu_d(\mathbb{P}\CC^{n+1}))$ is the Zariski closure in
$\mathbb{P}(\Sym^d\CC^{n+1})$ of the classes of points of Waring rank
less than or equal to $r$. In particular,
$\sigma_r(\nu_d(\mathbb{P}\CC^{n+1}))$ is given by some polynomial
equations, so if we know them, we can check if a given point is in the secant
variety. Unfortunately, these equations are hard to compute and are unknown in
general.

The paper \cite{LO13} presents many methods of obtaining equations vanishing on the
secant variety in the setting of vector bundles. However, the equations given
in this way are equations of a bigger variety, the so-called cactus variety. It is defined by
\begin{equation*}
\kappa_r(\nu_d(\PP \CC^{n+1})) = \overline {\bigcup_{R \hookrightarrow \PP \CC^{n+1}} \langle \nu_d(R)\rangle}\text{,}
\end{equation*}
where the union ranges over all zero-dimensional subschemes $R$ of length $r$. The overline denotes the Zariski closure, and $\langle \cdot \rangle$
denotes the linear span of a scheme.
See \cite{Gal16}, and the discussion in \cite{BB14} and \cite[\S 10.2]{Lan17}. In
fact, we are not aware of any explicit equation of the secant variety 
$\sigma_{r}(\nu_d(\mathbb{P} \CC^{n+1}))$ which does not vanish on the
respective cactus variety. Moreover, the cactus varieties fill up the
projective spaces much quicker than the secant varieties, see \cite{BJMR17}.

In this paper, we solve the problem of identification of points of the secant variety inside the cactus variety in
the minimal cases where these varieties differ. If $r < 14$, then the equality $\sigma_r(\nu_d(\PP \CC^{n+1})) = \kappa_r(\nu_d(\PP \CC^{n+1}))$ holds
for every $d, n \geq 1$ (\cite[Prop.~2.2]{BB14} and \cite[Thm. A]{CJN15}). That is why in Section \ref{s:14thsecant} we focus on studying
$\kappa_{14}(\nu_d(\PP\CC^{n+1}))$---the minimal cases when $\kappa_r(\nu_d(\PP \CC^{n+1}))$ can be reducible. We show that $\kappa_{14}(\nu_d(\PP
\CC^{n+1}))$ has two irreducible components for $d\geq 5,n \geq 6$, and we describe the one different from the secant variety. However, in order to give this description, we need to
introduce some notation.

We give the definitions for any algebraically closed field $\Bbbk$ since we
will need this generality in Theorems~\ref{t:general_polynomial_introduction} and \ref{t:general_subspace_introduction}. Fix a positive integer $n$
and let $T^*=\Bbbk[\alpha_0,\ldots,\alpha_n]$ be a polynomial ring with graded dual ring $T=\Bbbk_{dp}[x_0,\ldots,x_n]$, where the index $dp$ refers
to the divided power structure (see \cite[\S A2.4]{Eis95}). The duality between $T^*$ and $T$ can be written in the form
\[
  \alpha^\mathbf{u} \lrcorner x^{[\mathbf{v}]} = \begin{cases}
    x^{[\mathbf{v}-\mathbf{u}]} & \text{if } v_k\geq u_k \text{ for } k=0,\ldots,n\\
0 & \text{otherwise.}
\end{cases}
\]
In the above formula, $\mathbf{u}$ and $\mathbf{v}$ are multi-indices. Note that we use divided powers since they are a standard tool for
computing Waring rank over fields of non-zero characteristic. If $\operatorname{char} \kk = 0$, divided powers have a simple form. Namely, $x^{[\mathbf{v}]} = \frac{x^{\mathbf{v}}}{\mathbf{v}!}$, where $\mathbf{v}! = v_0!  v_1 ! \cdots v_{n}!$.

For ease of reference, in the following definition we collect a few notions and conventions.
\begin{definition}\label{d:basic}
Let $k$ be a positive integer, $R^*=\Bbbk[\alpha_0, \alpha_1,\ldots, \alpha_k]$ be a polynomial ring, and $R = \Bbbk_{dp}[x_0,x_1,\ldots, x_k]$ be its graded dual. Given a finite dimensional linear subspace $V\subseteq R$ we denote by $\Ann(V)$ the ideal $\Ann(V) = \{\theta\in R^* \mid \theta\lrcorner V = 0\}$.
We write $\Apolar(V)$ for the corresponding quotient ring $R^*/\Ann(V)$ and we call it the apolar algebra of $V$.  If $V=\langle f \rangle$ we write $\Ann(f)$ instead of $\Ann(\langle f \rangle)$ and $\Apolar(f)$ instead of $\Apolar(\langle f \rangle)$.
For a finite local $\Bbbk$-algebra $(A,\mathfrak{m})$ the (local) Hilbert function of $A$ is the Hilbert function of the associated graded ring $\operatorname{gr}_\mathfrak{m} A$ (\cite[\S5.1]{Eis95}).

For any graded ring $P$, by $P_i$ we denote the homogeneous part of degree $i$, and
\[
 P_{\leq i} = \bigoplus_{j\leq i} P_j.
\]
\end{definition}
\begin{definition}\label{d:prime_operator}
  
  Assume that $\Bbbk = \CC$. Then $R$ from Definition \ref{d:basic} becomes $\CC[x_0,\dots,x_k]$. Given positive integers $d$ and $m$  with $d\geq m$ and a linear subspace $W\subseteq R_{\leq m}$ we define
\[
  W^{\btd d} = \{(d-m)!F_m + (d-m+1)!F_{m-1} + \ldots + d!F_0 \text{, where } F_m+F_{m-1}+\ldots+ F_0 \in W \text{ and }F_j\in R_j\}.
\]

For a basis $(y_0,y_1,\dots,y_k)$ of $R_1$,  and a linear subspace $W \subseteq R$ we define $W|_{y_0=1}$ to be the dehomogenization of $W$ with
respect to that basis.
\end{definition}

We explain the need for $W^{\btd d}$ in Remark \ref{rem:why_triangle}.


It follows from \cite{CJN15} that for $n \geq 6$ and $d \geq 2$ the cactus variety $\kappa_{14}(\nu_d(\mathbb{P}\CC^{n+1}))$ has at most two irreducible
components. In general it can be irreducible. Consider for instance the case of $\nu_3 : \PP \CC^7 \to \PP (\Sym^3 \CC^7)$,
where the secant variety $\sigma_{14}(\nu_3(\PP \CC^7))$ fills the ambient space (which follows from the Alexander-Hirschowitz theorem).
In the following theorem in Part (i) we verify that for all $d\geq 5$ and all $n \geq 6$ the cactus variety $\kappa_{14}(\nu_d(\PP \CC^{n+1}))$ is reducible. The main result
is Part (ii) which describes the irreducible component other than $\sigma_{14}(\nu_d(\PP \CC^{n+1}))$.
\begin{theorem}\label{t:segre-veronese_map_general_n}
  Let $n\geq 6$ and $d\geq 5$ be integers and consider the polynomial ring $T=\CC[x_0,\ldots, x_n]$.
  \begin{enumerate}[label=(\roman*)]
  \item The cactus variety $\kappa_{14}(\nu_d(\PP T_1))$ has two irreducible components, one of which is $\sigma_{14}(\nu_d(\PP T_1))$, and we denote the other one by $\eta_{14}(\nu_d(\PP T_1))$.
  \item The irreducible component $\eta_{14}(\nu_d(\PP T_1))$ is the closure of the following set 
  \begin{multline*}%
    \{[y_0^{d-3} P] \in \mathbb{P}T_d\mid y_0 \in T_1\setminus \{0\}, [P] \in \mathbb{P}T_3 \text{, and there exists
    a completion of } y_0 \text{ to a basis } \\ (y_0,y_1,\ldots, y_n) \text{ of } T_1 \text{ such that } \Apolar((P|_{y_0=1})^{\btd d}) \text{ has
Hilbert function } (1,6,6,1)\}.
  \end{multline*}
  \end{enumerate}
\end{theorem}

For $n = 6$, Theorem \ref{t:segre-veronese_map_general_n} has the following simple form.

\begin{corollary}\label{c:segre-veronese_map}
  For $d \geq 5$, the cactus variety $\kappa_{14}(\nu_d(\mathbb{P}^6))$ has two irreducible components: the secant variety
  $\sigma_{14}(\nu_d(\mathbb{P}^6))$, and the variety $\eta_{14}(\nu_d(\mathbb{P}^6))$ consisting of degree $d$ forms divisible by the $(d-3)$-rd
  power of a linear form. 
\end{corollary}

\begin{remark}
  For $n > 6$, the variety $\eta_{14}(\nu_d(\mathbb{P}^n))$ is strictly contained in the set of projective classes of forms divisible by the
  $(d-3)$-rd power of a linear form. This follows from the fact that for a general $P \in T_3$, the Hilbert function of $\Apolar(P|_{y_0=1})$ is
  $(1,n,n,1)$. For a non-general $P \in T_3$, other Hilbert functions can occur, such as $(1,1,1,1), (1,2,2,1),\dots, (1,n-1,n-1,1)$, as well as some
  non-symmetric ones, such as $(1,7,5,1)$. Therefore, it is only for $n=6$ that $\eta_{14}(\nu_d(\PP^n))$ is the variety of forms divisible by
  $(d-3)$-rd power of a linear form.
\end{remark}

In the following theorem we present an algorithm which checks if a point in the cactus variety $\kappa_{14}(\nu_{d}(\mathbb{P}^n))$ is
in the secant variety $\sigma_{14}(\nu_{d}(\mathbb{P}^n))$ for $d \geq 6, n
\geq 6$. 
The case of $\PP^6$ is implemented in Macaulay2, see Appendix \ref{appendix}.

\begin{theorem}\label{t:algorithm}
Let $T=\mathbb{C}[x_0,\ldots,x_n]$ be a polynomial ring with $n\geq 6$. Given an integer $d\geq 6$ and $[G] \in \kappa_{14}(\nu_{d}(\mathbb{P}T_1)) \subseteq \mathbb{P}T_{d}$ the following algorithm checks if $[G]\in \sigma_{14}(\nu_{d}(\mathbb{P}T_1))$.
\begin{adjustwidth}{15pt}{0pt}
\begin{itemize}
\item[\textbf{Step 1}] Compute the ideal $\mathfrak{a} = \sqrt{((\Ann G)_{\leq d-3})}$.
\item[\textbf{Step 2}] If $\mathfrak{a}_1$ is not $n$-dimensional, then $[G]\in \sigma_{14}(\nu_{d}(\mathbb{P}T_1))$ and the algorithm terminates. 
Otherwise compute $ \{K\in T_1 \mid \mathfrak{a}_1 \lrcorner K = 0\}$. Let $y_0$ be a generator of this one dimensional $\mathbb{C}$-vector space.

\item[\textbf{Step 3}] Let $e$ be the maximal integer such that $y_0^e$ divides $G$. If $e\neq d-3$, then $[G] \in
  \sigma_{14}(\nu_{d}(\mathbb{P}T_1))$ and the algorithm terminates. Otherwise let $G=y_0^{d-3}P$, pick a basis $(y_0,y_1,\dots,y_n)$ of $T_1$ and compute $f=P|_{y_0=1}\in R:=\mathbb{C}[y_1,\dots,y_n]$.

\item[\textbf{Step 4}]  Let  $I=\Ann (f^{\btd d}) \subseteq R^*$.
  If the Hilbert function of $R^*/I$ is not equal to $(1,6,6,1)$, then $[G] \in \sigma_{14}(\nu_{d}(\PP T_1))$, and the algorithm terminates.
\item[\textbf{Step 5}]  Compute $r=\dim_{\mathbb{C}}\operatorname{Hom}_{R^*}(I, R^*/I)$. Then $[G] \in \sigma_{14}(\nu_{d}(\mathbb{P}T_1))$ if and only if $r>14n-8$.
\end{itemize}
\end{adjustwidth}
\end{theorem}

We prove Theorem \ref{t:algorithm} in Section \ref{s:14thsecant}. 

The notion of the Waring rank of a homogeneous polynomial can be generalized to the rank of a subspace which corresponds to the problem of finding a minimal simultaneous decomposition of many forms as sums of powers of the same set of linear forms. This leads to a generalization of secant varieties to the notion of a Grassmann secant variety $\sigma_{r,k}(\nu_d(\PP \CC^{n+1}))$ for positive integers $r,k,d,n$ (see Section \ref{s:Apolarity_lemmas} for the definitions). We have $\sigma_{r,1}(\nu_d(\PP \CC^{n+1})) = \sigma_r(\nu_d(\PP \CC^{n+1}))$ as defined above. For a Grassmann secant variety, there is an analogous notion of a Grassmann cactus variety $\kappa_{r,k}(\nu_d(\PP \CC^{n+1}))$ (see Section \ref{s:Apolarity_lemmas}).

In this paper, we solve the problem of identification of points of the Grassmann secant variety inside the Grassmann cactus variety in the minimal
cases where these varieties differ. If $r < 8$, then the equality $\sigma_{r,k}(\nu_d(\PP \CC^{n+1})) = \kappa_{r,k}(\nu_d(\PP \CC^{n+1}))$ holds for every
$d, n ,k\geq 1$ (\cite[Thm. 1.1]{CEVV09}). That is why in Section~\ref{s:83grassmann} we focus on studying
$\kappa_{8,3}(\nu_d(\PP\CC^{n+1}))$---the minimal case when $\kappa_{r,k}(\nu_d(\PP \CC^{n+1}))$ can be reducible (see Remark~\ref{r:why_83} for the
reason why we study the particular case of $\kappa_{8,3}(\nu_d(\PP \CC^{n+1}))$). Theorem~\ref{t:segre-veronese_map_143_general_case}, which is proven in Section~\ref{s:83grassmann}, is a counterpart of Theorem~\ref{t:segre-veronese_map_general_n} in
the Grassmann case.

In Theorem~\ref{t:algorithm_143} we present an algorithm analogous to Theorem~\ref{t:algorithm} which checks if a point in the Grassmann cactus variety $\kappa_{8,3}(\nu_{d}(\mathbb{P}^n))$ is
in the Grassmann secant variety $\sigma_{8,3}(\nu_{d}(\mathbb{P}^n))$ for $d \geq 5$, $n\geq 4$. 

On our way to establishing Theorems  \ref{t:segre-veronese_map_general_n} and \ref{t:segre-veronese_map_143_general_case}, we prove Theorems \ref{t:general_polynomial_introduction} and \ref{t:general_subspace_introduction}, which determine the border cactus rank of forms and subspaces divisible by a large power of a linear form, and which can be applied in more general settings.

Let $S^* = \kk[\alpha_1\dots,\alpha_n] \subseteq T^*$ (we omit the variable $\alpha_0$ from $T^*)$. Then the graded dual ring $S = \kk_{dp}[x_1,\dots,x_n]$ is naturally a subring of $T$. 

\begin{definition}\label{d:homogenization_of_space}
Let $d_1 \geq 1, d_2 \geq 0$ be integers.
Given a linear subspace $W\subseteq S_{\leq d_1}$,
define a linear subspace
\[W^{hom, d_2} = \left\{ \sum_{i=0}^{\deg f} F_i x_0^{[d_2+d_1-i]} \mid f=F_{\deg f} + \ldots + F_0\in W, \text{ where } F_i \in S_i \right\} \subseteq T_{d_1+d_2}.
\]
If $W = \langle f \rangle$, we write $f^{hom, d_2}$ instead of $\langle f\rangle^{hom,d_2}$.
\end{definition}

In the statements of Theorems \ref{t:general_polynomial_introduction},
\ref{t:general_subspace_introduction} below we will use the notion of the
cactus rank and the border cactus rank defined in Section
\ref{s:Apolarity_lemmas}.  
\begin{theorem}[Polynomial case]\label{t:general_polynomial_introduction}
 Let $f \in S_{\leq d_1}, f = F_{d_1} +\dots+F_0$, where $F_i \in S_i$
and $r = \dim_{\kk} S^*/\Ann(f)$. Assume that $F_{d_1}$ is not a power of a linear
 form. Then  we have the following:
\begin{itemize}
  \item[(i)] The cactus rank of $f^{hom,d_2}$ is not greater than $r$.
  \item[(ii)] If $d_2\geq d_1-1$, then the border cactus rank  of $f^{hom,d_2}$ equals $r$. In particular, the cactus rank and border cactus rank of
    $f^{hom,d_2}$ are equal. 
\end{itemize}
\end{theorem}

\begin{theorem}[Subspace case]\label{t:general_subspace_introduction} 
 Let $W \subseteq S_{\leq d_1}$, and $r = \dim_\kk S^*/\Ann(W)$. We have the following:
\begin{itemize}
\item[(i)] The cactus rank of $W^{hom, d_2}$ is not greater than $r$.
\item[(ii)] If $d_2\geq d_1$, then the border cactus rank of $W^{hom, d_2}$  equals $r$. In particular, the cactus rank and border cactus rank of
  $W^{hom, d_2}$ are equal.
\end{itemize}
\end{theorem}

Additionally, we show more or less the uniqueness of the border cactus decomposition (see Theorems \ref{t:general_polynomial},  \ref{t:general_subspace} for more precise statements).

\begin{remark}\label{rem:why_triangle}
Assume that $\Bbbk = \CC$. Then the ring $S$ is canonically isomorphic to the polynomial ring $\CC[x_1,x_2 \ldots, x_n]$.
Therefore, given a polynomial $f\in S_{\leq d_1}$, we may consider $f^{hom}$, i.e.\ its homogenization  with respect to $x_0$ in $T=\CC[x_0, x_1, \ldots, x_n]$. Then for any non-negative integer $d_2$ we have the equality
\[
  (f^{\btd d_1 + d_2})^{hom, d_2} = x_0^{d_2+d_1-\deg(f)}f^{hom}.
\]
Since Theorem \ref{t:general_polynomial_introduction} is one of our main tools, and the $()^{hom, d_2}$ operator appears there, we consider the triangle operator.
\end{remark}

The structure of the paper is as follows. In Section \ref{s:Apolarity_lemmas} we recall the definitions of different kinds of rank of a subspace of a polynomial ring. We also discuss the related versions of Apolarity Lemma. In Section~\ref{s:algebraic_results} we prove a series of algebraic results which will be needed later. These are mainly about Hilbert functions and annihilator ideals. Section~\ref{s:general_results} is devoted to proving Theorems \ref{t:general_polynomial_introduction} and \ref{t:general_subspace_introduction}.
In Section~\ref{s:14thsecant} we study $\kappa_{14}(\nu_d(\PP^n))$ for $n\geq 6$ and $d\geq 5$. Then we describe the irreducible components of the Grassmann cactus variety $\kappa_{8,3}(\nu_d(\PP^n))$ for $n\geq 4$ and $d\geq 5$, see Section \ref{s:83grassmann}. Proofs of our main results use the existence of a certain morphism to the Hilbert scheme. Since its proof is technical, we defer it to Appendix \ref{s:hilbert}. Finally, in Appendix \ref{appendix} we present an implementation of the algorithm from Theorem \ref{t:algorithm}.

\subsection{Acknowledgments}
We thank Jarosław Buczyński for introduction to this subject and constant support. We are grateful to Jarosław Buczyński and Joachim Jelisiejew for
many discussions, in particular suggestions on how to improve the presentation. Piotr Achinger and Francesco Galuppi read our article and provided us
with helpful comments.

This project originated as an aftermath of the secant variety working group of the Simons Semester \textit{Varieties: Arithmetic and Transformations} at the Banach Center (\url{https://www.impan.pl/~vat/}). The event was supported by the grant 346300 for IMPAN from the Simons Foundation and the matching 2015-2019 Polish MNiSW fund.

Gałązka is supported by the National Science Center, Poland, project number 2017/26/E/ST1/00231.

Mańdziuk is supported by the National Science Center, Poland, projects number 2017/26/E/ST1/00231 and 2019/33/N/ST1/00858.

Rupniewski is supported  by the National Science Center, Poland, projects number 2017/26/E/ST1/00231 and 2019/33/N/ST1/00068.

\section{Ranks and Apolarity Lemmas}\label{s:Apolarity_lemmas}
One of the main tools in studying certain notions of rank is the Apolarity Lemma in its many variants. 
It translates geometric questions into algebraic problems of existence of some ideals. Not only can it be applied for
establishing upper bounds for rank by constructing certain ideals, but it also can provide lower bounds. The latter is done 
by proving the non-existence of ideals satisfying given properties.
Some of the many examples of applying Apolarity Lemmas in both directions are \cite[Thm.~1.5]{Gal20}, \cite[\S 4]{BB15}, as well as Theorems \ref{t:general_subspace}, \ref{t:general_polynomial}.

In this section we recall the definitions of various types of ranks and the corresponding variants of the Apolarity Lemma. 
Subsections \ref{ss:rank_and_border_rank} and \ref{ss:cactus_rank_and_cactus_border_rank} introduce ranks of subspaces and corresponding geometric objects, Grassmann secant and Grassmann cactus varieties.

The problem of decomposing many forms simultaneously as sums of powers of the same set of linear forms and the connected notion of Grassmann secant variety originates from the work of Terracini \cite{Ter1915} and was later studied by Bronowski \cite{Bro33}. The paper \cite{BL13} investigates the relation between the ranks of tensors in the Segre embedding of 3 copies of projective space and the ranks of subspaces in the Segre embedding of 2 copies of projective space. The problem of defectivity of Grassmann secant varieties is addressed in  \cite{BBCC13}, \cite{CC08}, \cite{Fon02}. Simultaneous decomposition of forms of different degrees is studied among others in \cite{AGMO18} and \cite{CV17}.

We will use the following notation.
Let $n$ be a positive integer, $T^* = \kk [\alpha_0, \alpha_1, \dots ,\alpha_n ]$ be a polynomial ring over an algebraically closed field $\kk$ and $T$ be the graded dual ring of $T^*$.
Let $\nu_d : \mathbb{P}T_1 \to \mathbb{P}T_d$, $[L] \mapsto [L^{[d]}]$ be the $d$-th Veronese map. 
Given a subscheme $R \subseteq \PP T_d$ we denote by $\langle R \rangle$  the projective linear span in $\PP T_d$ of the scheme $R$,
  i.e.\ the smallest projective linear subspace of $\PP T_d$ containing the scheme $R$.
  For  a subset $Y$ in a variety $X$, by $\overline{Y}$ we denote the Zariski closure of $Y$ in $X$.
 Let $\Gr(k, V)$ be the Grassmannian of $k$-dimensional subspaces of a linear space $V$.
We use these notations in the whole section. 

We introduce the following definition, which will be used in the next subsection.

\begin{definition}\label{d:standard_hilbert_function}
For a positive integer $s$, let
$h_s\colon\mathbb{Z}_{\geq 0}\to\mathbb{Z}_{\geq 0}$ be given by $h_s(a)=\min\{\dim_\Bbbk T^*_a,s\}$. Notice that $h_s$ depends on $n$. More generally, a function $h \colon \ZZ_{\geq 0} \to \ZZ_{\geq 0}$ satisfying the following conditions will be called an $(s,n+1)$-standard Hilbert function:
\begin{enumerate}[label=(\alph*)]
 \item $h(d) \leq h (d + 1)$ for all $d$,
 \item if $h(d) = h(d+1)$, then $h(e) = s$ for all $e \geq d$,
 \item $0 \leq h(d) \leq h_s (d)$ for all $d$.
\end{enumerate}
\end{definition}

\subsection{Rank and border rank}\label{ss:rank_and_border_rank}

\begin{definition}
  The rank of a $k$-dimensional linear subspace $V$ of $T_d$  is
  \[
    \rr(V) = \min\{r \in \mathbb{Z}_{>0} \mid \PP V \subseteq \langle L_1^{[d]} , \dots , L_r^{[d]} \rangle\text{ for some } L_i \in T_1\}\text{.}
  \]
  The $(r,k)$-th Grassmann secant variety of the $d$-th Veronese variety is
  \begin{equation*}
    \sigma_{r,k}(\nu_d(\mathbb{P}T_1)) = \overline{\{[V] \in \Gr (k,T_d) \mid \rr(V) \leq r \}}\text{.}
  \end{equation*}

\noindent  The border rank of $V$ is
  \[
    \brr(V) = \min\{r \in \mathbb{Z}_{>0} \mid [V] \in \sigma_{r,k}(\nu_d(\mathbb{P}T_1))\}\text{.}
  \]
\end{definition}

If $k=1$, i.e. if $V=\langle F \rangle$ for an element $F\in T_d$ we obtain the classical notions of rank and border rank of $F$ and the secant variety $\sigma_r(\nu_d(\PP T_1))$.

Note that border rank is more natural from the point of view of algebraic geometry than rank, since it provides a condition for $[F]$ to be a point of a Zariski closed subset of $\PP T_d$. However, the variant of apolarity for border rank has been stated only very recently (see \cite{BB19} for the case $V=\langle F \rangle$ and \cite{BB20} for the general case). Nevertheless, it has already been applied in \cite{CHL19},\cite{HMV20}, and \cite{Man20}.

In order to formulate this result, we have to consider multigraded Hilbert schemes, see \cite{HS04}. Denote by
$\operatorname{Hilb}_{T^*}^{h_r}$ the multigraded Hilbert scheme associated 
with the polynomial ring $T^*$ (with the standard $\mathbb{Z}$-grading) and the function $h_r$, as defined in Definition \ref{d:standard_hilbert_function}. Let $\Slip_{r, \mathbb{P}T_1}$ be the closure in
$\operatorname{Hilb}_{T^*}^{h_r}$ of points corresponding to saturated ideals of $r$ points. Then the following holds.

\begin{proposition}[Border Apolarity Lemma]\label{p:border_apolarity}
  Let $V \subseteq T_d$ be a $k$-dimensional subspace. Then $\brr(V) \leq r$ if and only if there exists an ideal $[I] \in \Slip_{r,\mathbb{P}T_1}$ such that
  \[
    I \subseteq \Ann(V)\text{.}
  \]
\end{proposition}
\begin{proof}
This follows from the proof of \cite[Thm. 1.3]{BB20} if we specify $\mathcal{H} = \Hilb^{sm}_r(\PP^n)$, the smoothable component of the Hilbert scheme. Here are the details.
  We know that
  \begin{align*}
    \brr(V) \leq r &\iff [V] \in \sigma_{r,k}(\nu_d(\mathbb{P}^n)) 
    \iff \exists_{[I] \in \Slip_{r,\mathbb{P}T_1}} I_d\subseteq V^\perp
  \end{align*}
  where $V^\perp$ is the subspace of $T^*_d$ of forms annihilating $V$. The
  latter equivalence follows from \cite[Prop. 6.1]{BB20}.
  We need to prove that
  \[
   \exists_{[I] \in \Slip_{r,\mathbb{P}T_1}} I_d\subseteq V^\perp 
    \iff \exists_{[I] \in \Slip_{r,\mathbb{P}T_1}} I\subseteq \Ann(V) .
  \]
  One implication is clear. We show the implication from the left to the right. 
  Let $\phi \in I_{e}$ for $e \in \ZZ$, then $T_{d-e}^* \cdot \phi   \in I_d \subset V^\perp \subset \Ann(V)$. Thus ($T_{d-e}^* \cdot \phi)  \lrcorner V = 0$, which implies $T_{d-e}^* \lrcorner (\phi  \lrcorner V) = 0$. We obtain $\phi  \lrcorner V = 0$.   
\end{proof}

\subsection{Cactus rank and border cactus rank}\label{ss:cactus_rank_and_cactus_border_rank}
\begin{definition}\label{d:cactus_rank}
  The cactus rank of a $k$-dimensional linear subspace $V$ of $T_d$  is
  \[
    \crr(V) = \min\{r \in \mathbb{Z}_{>0} \mid \PP V \subseteq \langle \nu_d(R) \rangle \text{ for a zero-dimensional subscheme } R \subseteq \mathbb{P}T_1,\length R = r \}\text{.}
  \]
  The $(r,k)$-th Grassmann cactus variety of the $d$-th Veronese variety is
  \begin{equation*}
    \kappa_{r,k}(\nu_d(\mathbb{P}T_1)) = \overline{\{[V] \in \Gr (k,T_d) \mid \crr(V) \leq r \}}\text{.}
  \end{equation*}

  \noindent The border cactus rank of $V$ is
  \[
    \bcrr(V) = \min\{r \in \mathbb{Z}_{>0} \mid [V] \in \kappa_{r,k}(\nu_d(\mathbb{P}T_1))\}\text{.}
  \]
\end{definition}

\begin{proposition}[Cactus Apolarity Lemma]\label{p:cactus_apolarity}
 Let $V \subseteq T_d$ be a non-zero subspace and $I(R)$  be the saturated ideal of a subscheme $R  \subseteq  \PP T_1$. Then
  \[
    I(R) \subseteq \Ann(V) \iff \PP V \subseteq  \langle \nu_d(R) \rangle\text{.}
  \]
  Therefore $\crr(V) \leq r$ if and only if there exists a zero-dimensional subscheme $R \subseteq \mathbb{P}T_1$ of length $r$ such that
  \[
    I(R) \subseteq \Ann(V) \text{.}
  \]
\end{proposition}

\noindent For a proof, see \cite[Thm.~4.7]{Tei14}. Due to Proposition \ref{p:cactus_apolarity}, the cactus rank of $V$ could be defined as the smallest length of a scheme $R\subseteq \PP T_1$ such that $I(R)\subseteq \Ann(V)$. This approach is taken for instance in \cite{BR13} or \cite[Def.~5.1]{IK06}.

In order to formulate a version of apolarity for border cactus rank,  we need to consider all $(r, n+1)$-standard Hilbert functions (see Definition \ref{d:standard_hilbert_function}), instead of $h_r$ as in Border Apolarity Lemma \ref{p:border_apolarity}.

\begin{proposition}[Weak Border Cactus Apolarity Lemma]\label{p:weak_border_cactus_apolarity_lemma}
 Let $V \subseteq T_d$ be a non-zero subspace. If $\bcrr (V) \leq r$, then there exists a homogeneous ideal $I \subseteq \Ann(V) \subseteq T^*$ such that $T^*/I$ has an  $(r, n+1)$-standard Hilbert function.
\end{proposition}
\noindent See \cite[Thm. 1.1]{BB20} for a proof.

\section{Algebraic results}\label{s:algebraic_results}

In this section we present some algebraic results which will be needed in
Section \ref{s:general_results}. We will use the following
notation. Fix a positive integer $n$ and let
$S^*=\Bbbk[\alpha_1,\ldots,\alpha_n]\subseteq
T^*=\Bbbk[\alpha_0,\ldots,\alpha_n]$ be polynomial rings over an algebraically closed field $\kk$ with graded dual rings
$S=\Bbbk_{dp}[x_1,\ldots,x_n]\subseteq T=\Bbbk_{dp}[x_0,\ldots,x_n]$. We shall also use notations from Definitions \ref{d:prime_operator} and \ref{d:homogenization_of_space}.

If $J \subseteq S^*$ is an ideal, we denote by $J^{hom} \subseteq T^*$ its 
homogenization with respect to $\alpha_0$, see \cite[\S 8.4]{CLO}. 
If $J\subseteq T^*$ is an ideal, we denote by $J^{sat}$ its saturation with respect 
to the ideal $(\alpha_0,\dots,\alpha_n)$, see \cite[\S 15.10.6]{Eis95}.
If $M$ is a $\ZZ$-graded module, and $e$ is an integer, we denote by $H(M,e)$ the
Hilbert function of $M$ at $e$. 

The following lemma says that the homogenization in $T^*$ of an ideal in $S^*$ is saturated. This will enable us to use  Cactus Apolarity Lemma \ref{p:cactus_apolarity} in the proofs of Theorems \ref{t:general_polynomial_introduction} and \ref{t:general_subspace_introduction}.

\begin{lemma}\label{l:homogenization_is_saturated} Let $I\subseteq S^*$ be an
  ideal. Then the homogenization $I^{hom}\subseteq T^*$ is saturated with respect to the
  irrelevant ideal $(\alpha_0,\ldots, \alpha_n)$.
\end{lemma}
\begin{proof}
  It is enough to show that $\{\theta \in S^* \mid \theta \alpha_0 \in
  I^{hom} \} \subseteq I^{hom}$. Take $\theta \in S^*$ such that $\theta
  \alpha_0 \in I^{hom}$. We may assume that $\theta$ is homogeneous. By assumption, for some integer $s$, there are $\zeta_1,\ldots,\zeta_s\in I$ and
  $\xi_1,\ldots,\xi_s\in T^*$ such that

\[
\theta\alpha_0=\xi_1\zeta_1^{hom}+\ldots +\xi_s\zeta_s^{hom}.
\]
Hence $\theta_{\mid_{\alpha_0=1}} \in I$, so $(\theta_{\mid_{\alpha_0=1}})^{hom} \in I^{hom}$. Thus 
\[
\theta = \alpha_0^{\deg \theta - \deg(\theta_{\mid_{\alpha_0=1}})}(\theta_{\mid_{\alpha_0=1}})^{hom} \in I^{hom},
\]
as claimed.
We used \cite[Prop. 8.2.7 (iii) and (iv)]{CLO}.
\end{proof}

If $I \subseteq S^*$ is a homogeneous ideal, then there is a simple way to calculate the Hilbert function of $T^*/I^{hom}$ from the Hilbert function of $S^*/I$, namely
\[
 H(T^*/I^{hom}, e)= \sum_{i=0}^e H(S^*/I, i) \text{ for } e \in \ZZ_{\geq 0}.
\]
In particular if the Hilbert polynomial of $S^*/I$ is zero, then for large enough $e$ we have $ H(T^*/I^{hom}, e) = \dim_\kk S^*/I$. Lemma \ref{l:Hilbert_function_of_homogenization} and Corollary \ref{c:HF_of_extension_stabilizes} generalize this to inhomogeneous ideals.

In the following lemma we use the definitions
of monomial orders $<$, leading terms $\operatorname{LT}_<$, leading monomials
$\operatorname{LM}_<$, and Gr\"obner bases as in \cite[Ch. 2]{CLO}.

\begin{lemma}\label{l:Hilbert_function_of_homogenization} Let $I\subseteq S^*$ be an  ideal. Let $<$ be any monomial order on $S^*$ that respects the degree. Then for every non-negative integer $e$
\[
H(T^*/I^{hom}, e) = \#\{\mu\in S^* \mid \mu \emph{ is a monomial, } \deg \mu \leq e \emph{ and } \mu\notin \operatorname{LT}_<(I)\}.
\]
\end{lemma}
\begin{proof}
Let $e\in \mathbb{Z}_{\geq 0}$ and consider sets 
\[
A_{\leq e} = \{\mu\in S^* \mid \mu \text{ is a monomial, } \deg \mu \leq e\}
\]
and 
\[
B_e = \{\mu\in T^* \mid \mu\text{ is a monomial, }\deg \mu = e\}.
\]
These sets are in bijection given by
\[
A_{\leq e} \ni \mu \mapsto \alpha_0^{e-\deg \mu}\mu \in B_e
\]
and
\[
B_e \ni \mu \mapsto \mu|_{\alpha_0=1} \in A_{\leq e}.
\]
Let $<_h$ be the monomial order on $T^*$ defined by $\alpha_0^{a_0}\ldots\alpha_n^{a_n}<_h \alpha_0^{b_0}\ldots\alpha_n^{b_n}$ if and only if $\alpha_1^{a_1}\ldots\alpha_n^{a_n} < \alpha_1^{b_1}\ldots\alpha_n^{b_n}$ or $\alpha_1^{a_1}\ldots\alpha_n^{a_n}=\alpha_1^{b_1}\ldots\alpha_n^{b_n}$ and $a_0 < b_0$. We have $H(T^*/I^{hom}, e) = \# \{\mu\in B_e \mid  \mu\notin \operatorname{LT}_{<_h}(I^{hom})\}$ (see \cite[Thm. 15.3]{Eis95}). Therefore it is enough to show that for $\mu\in A_{\leq e}$ we have $\mu\in \operatorname{LT}_<(I)$ if and only if $\alpha_0^{e-\deg \mu}\mu\in \operatorname{LT}_{<_h}(I^{hom}).$

Assume that $\mu\in A_{\leq e}\cap \operatorname{LT}_<(I)$ and let $\theta\in I$ be such that $\operatorname{LM}_<(\theta) = \mu$. Then $\alpha_0^{e-\deg \mu}\theta^{hom} \in I^{hom}$ and $\operatorname{LM}_{<_h}(\alpha_0^{e-\deg \mu}\theta^{hom}) = \alpha_0^{{e-\deg \mu}} \operatorname{LM}_{<_h}(\theta^{hom}) = \alpha_0^{{e-\deg \mu}}\mu$. The latter equality follows from the following observation:
 for $\theta\in S^*$ we have $\operatorname{LM}_{<_h}(\theta^{hom}) = \operatorname{LM}_{<}(\theta)$.

Now assume that $\alpha_0^{{e-\deg \mu}}\mu\in B_e\cap \operatorname{LT}_{<_h}(I^{hom})$. Let $G=\{\zeta_1,\ldots,\zeta_k\}$ be a Gr\"obner basis for
$I$ with respect to $<$. Then $G^{hom}=\{\zeta_1^{hom},\ldots,\zeta_k^{hom}\}$ is a Gr\"obner basis for $I^{hom}$ with respect to $<_h$ (see
\cite[Thm. 8.4.4]{CLO}). Therefore for some $j\in\{1,\dots,k\}$ the monomial $\operatorname{LM}_{<_h}(\zeta_j^{hom}) = \operatorname{LM}_{<}(\zeta_j)$ divides $\alpha_0^{{e-\deg \mu}}\mu$. Thus, $\operatorname{LM}_{<}(\zeta_j)$ divides $\mu$.
\end{proof}

The following corollary of Lemma \ref{l:Hilbert_function_of_homogenization} will be used extensively. It shows that the Hilbert polynomial of the subscheme defined by $\Ann(W)^{hom}$ is equal to $\dim_\Bbbk S^*/\Ann(W)$. Moreover it provides an upper bound on the minimal degree from which the Hilbert function agrees with the Hilbert polynomial.

\begin{corollary}\label{c:HF_of_extension_stabilizes}
Let $W\subseteq S_{\leq d_1}$ be a linear subspace. Then for $e\geq d_1$
\[
H(T^*/\Ann(W)^{hom}, e) = \dim_\Bbbk S^*/\Ann(W).
\]
\end{corollary}
\begin{proof}
All monomials of degree at least $d_1+1$ are in $\Ann(W)$. Therefore for $e\geq d_1$ it follows from  Lemma~\ref{l:Hilbert_function_of_homogenization} that
\[H(T^*/\Ann(W)^{hom}, e) = H(T^*/\Ann(W)^{hom}, d_1)\]
is equal to the number of monomials in $S^*$ that do not belong to $\operatorname{LT}_<(\Ann(W))$. This number is the dimension of the quotient algebra $S^*/\Ann(W)$ as a $\Bbbk$-vector space (\cite[Thm. 15.3]{Eis95}). 
\end{proof}

The following result is similar to Lemma \ref{l:Hilbert_function_of_homogenization}. It compares the  Hilbert functions of two related quotient algebras, one of $S^*$ and one of $T^*$.
We will use it in the proof of Part (iii) of Theorem \ref{t:general_polynomial}.
\begin{lemma}\label{l:hilbert_function_and_contraction}
Let $J\subseteq T^*$ be a homogeneous ideal and $\theta=\alpha_0^d+\rho$ be an element of $J_d$ with $\rho$ of degree smaller than $d$ with respect to $\alpha_0$. Consider the contraction $J^c = J\cap S^*$. 
Then for any integer $e$ we have
\[
H(T^*/J, e) \leq H (S^*/J^c,e)+H(S^*/J^c, e-1)+\ldots +H(S^*/J^c, e-d+1). 
\]
\end{lemma}
\begin{proof}
Let $<$ be the graded lexicographic order on $T^*$ with $\alpha_n<\alpha_{n-1} < \ldots < \alpha_0$ and consider its restriction $<'$ to $S^*$. It follows from \cite[Thm. 15.3]{Eis95}) that $H(T^*/J, e)$ is the number of monomials of degree $e$, not in $\operatorname{LT}_<(J)$. Observe that every monomial divisible by $\alpha_0^d$ is in $\operatorname{LT}_<(J)$. Therefore, we have
\[
H(T^*/J, e) = \sum_{i=0}^{d-1} \#\{\mu\in S^*_{e-i} \mid \mu \text{ is a monomial and } \alpha_0^i\mu\notin \operatorname{LT}_<(J)\}.
\]
Fix $0\leq i \leq d-1$ and let $\mu$ be a monomial of degree $e-i$ from $S^*$.  If $\mu\in \operatorname{LT}_{<'}(J^c)$, then there is a homogeneous $\zeta\in J^c$ such that $\operatorname{LM}_{<'}(\zeta) = \mu$. Therefore, $\alpha_0^i\zeta \in J$ and $\operatorname{LM}_{<}(\alpha_0^i\zeta) = \alpha_0^i\mu$. Thus for $i\in \{0,\ldots,d-1\}$ we have 
\[
\#\{\mu\in S^*_{e-i} \mid \mu \text{ is a monomial and } \alpha_0^i\mu\notin \operatorname{LT}_<(J)\} \leq H(S^*/J^c, e-i).
\]
\end{proof}

In \cite[Prop. 1.6]{BBKT15} it is proven that annihilator of a homogeneous degree $d$ polynomial that is not a power of a linear form has a set of minimal generators of degrees at most $d$. The following lemma generalizes it to  inhomogeneous polynomials.

\begin{lemma}\label{l:extension_of_annihilator_generated_in_low_degrees} 
  Let $f=F_{d_1}+F_{d_1-1}+ \ldots + F_0$ be a polynomial of degree $d_1 \geq 2$ in
  $S$ where $F_i\in S_i$. Assume that $F_{d_1}$ is not a power of a linear form.
  Then $\Ann(f)^{hom}\subseteq T^*$ has a set of minimal generators of degrees not
  greater than $d_1$.
\end{lemma} 
\begin{proof}
We have $\Ann(f)\supseteq S^*_{d_1+1}$ so we may choose a set of its generators of the form
\[
  \Ann(f)=(\{\alpha^\mathbf{u} \mid \mathbf{u}\in \mathbb{Z}_{\geq 0}^n \text{ s.t. } |\mathbf{u}| =d_1+1\})+(\zeta_1,\dots,\zeta_k) \text{ with } \deg(\zeta_i)\leq d_1.
\]

Using Buchberger's algorithm for this set of generators and grevlex monomial order, we obtain a Gr\"obner basis of $\Ann(f)$ of the form
\begin{equation}\label{eq:buchberger}
  \{\alpha^\mathbf{u} \mid \mathbf{u}\in \mathbb{Z}_{\geq 0}^n \text{ s.t. } |\mathbf{u}| =d_1+1\} \cup \{\zeta_1,\dots,\zeta_k\} \cup \{\xi_1,\dots,\xi_l \}.
\end{equation}
We claim that $\deg \xi_i \leq d_1$. Let $\mathcal{G} = \{\alpha^\mathbf{u} \mid \mathbf{u}\in \mathbb{Z}_{\geq 0}^n \text{ s.t. } |\mathbf{u}| =d_1+1\}$. Note that each $S$-polynomial considered in the Buchberger's algorithm is divided with remainder by a set of polynomials containing $\mathcal{G}$. Therefore, $S$-polynomials of degree at least $d_1+1$ do not give new elements  of the Gr\"obner basis.

The ideal $\Ann(f)^{hom}$ is generated by the homogenizations of the
elements in Equation \eqref{eq:buchberger} (\cite[Thm. 8.4.4]{CLO}).  It is
enough to show that we can replace the monomial generators of degree $d_1+1$
written above by some generators of degree not greater than $d_1$.  Let
$\mathbf{u}\in \mathbb{Z}_{\geq 0}^n$ with $|\mathbf{u}|=d_1+1$. Then in $S^*$,
we can write $\alpha^\mathbf{u} = \sum_{i=1}^m \delta_i\gamma_i$ for some
$\delta_i\in \Ann(F_{d_1})_{d_1}$ and $\gamma_i\in S^*_1$ (\cite[Prop.
1.6]{BBKT15}). We have $\delta_i\in \Ann(f)$ for degree reasons. Therefore
$\alpha^\mathbf{u} \in ((\Ann(f)^{hom})_{\leq d_1})$ as an element of $T^*$.
\end{proof}

For a homogeneous polynomial $F_{d_1} \in S_{d_1}$ of positive degree, $\Ann(F_{d_1}x_0^{[d_2]}) = (\alpha_0^{d_2+1})+\Ann(F_{d_1})^{hom}$. In particular, $\Ann(F_{d_1}x_0^{[d_2]})_{\leq d_2} = (\Ann(F_{d_1})^{hom})_{\leq d_2}$. Lemma \ref{l:ideals_agree_in_low_degree} generalizes it to an arbitrary polynomial. Part (i) was proven in \cite[Lem. 2]{BR13}. However, from the
notation of the authors it is not clear that they use divided powers, but they
are essential for the lemma to work (see Example \ref{e:divided_powers}). For
this reason we present their proof with explicit use of divided powers.

Recall the notation of $f^{hom,d_2}$ from Definition \ref{d:homogenization_of_space}.

\begin{lemma}\label{l:ideals_agree_in_low_degree} 
  Let $f=F_{d_1}+F_{d_1-1}+\ldots+F_0$ be a degree $d_1 \geq 1$ polynomial in $S$ and $r=\dim_\Bbbk
  S^*/\Ann(f)$.
Let  $d_2$ be a non-negative integer. 
We have
\begin{itemize}
\item[(i)] $     \Ann(f)^{hom} \subseteq \Ann(f^{hom,d_2})\text{.}
  $
\item[(ii)] $(\Ann(f)^{hom})_{\leq d_2} = \Ann(f^{hom,d_2})_{\leq d_2} $.
\item[(iii)] If $d_2=d_1-1$, then $H(T^*/\Ann(f^{hom,d_2}), d_1) = r$ or $H(T^*/\Ann(f^{hom,d_2}),
  d_1)=r-1$. Moreover, in the latter case $\Ann(f^{hom,d_2}) =
  (\alpha_0^{d_1}+\rho)+\Ann(f)^{hom}$, where $\rho\in T^*_{d_1}$ has degree
  smaller than $d_1$ with respect to $\alpha_0$.
\end{itemize}
\end{lemma}
\begin{proof}

The proof of the lemma is based on the following calculation.
  Let $\Gamma=\alpha_0^d\Theta_0+\alpha_0^{d-1}\Theta_1+ \ldots +\Theta_d$, where $\Theta_i\in S^*_i$. We can rewrite $\Gamma \lrcorner f^{hom,d_2}$ as follows
  \begin{align}\label{e:contracting}\begin{split}
    \Gamma \lrcorner f^{hom,d_2} &= \sum_{e=0}^{d_1} \sum_{j=0}^{\operatorname{min}(d_1-e, d)} (\alpha_0^{d-j}\Theta_j) \lrcorner (x_0^{[d_1+d_2-(e+j)]}F_{e+j}) \\
    &= \sum_{e=0}^{d_1} \sum_{j=0}^{\operatorname{min}(d_1-e, d)} (\alpha_0^{d-j}\lrcorner x_0^{[d_1+d_2-(e+j)]}) (\Theta_j\lrcorner F_{e+j}) \\
    &= \sum_{e=0}^{\min(d_1,d_1+d_2-d)} \sum_{j=0}^{\operatorname{min}(d_1-e, d)} x_0^{[d_1+d_2-d-e]} (\Theta_j\lrcorner F_{e+j}) \\
    &=\sum_{e=0}^{\min(d_1,d_1+d_2-d)}x_0^{[d_1+d_2-d-e]}\sum_{j=0}^{\min(d_1-e, d)}\Theta_j\lrcorner F_{e+j}\text{.}
\end{split}\end{align}
\begin{itemize}
\item[(i)]
Let $\theta = \Theta_0 + \dots + \Theta_d \in \Ann(f)$, where $\Theta_i$ is  homogeneous of degree $i$. We show that $\theta^{hom} =
  \alpha_0^d\Theta_0+\alpha_0^{d-1}\Theta_1+ \ldots +\Theta_d$ is in the annihilator
  of $f^{hom,d_2}$. We put $\Gamma=\theta^{hom}$ in Equation \eqref{e:contracting}.

For every $e =0,\dots,{\min(d_1, d_1+d_2-d)}$ the sum $\sum_{j=0}^{\min(d_1-e, d)}\Theta_j\lrcorner
F_{e+j}$ is zero since $\theta\lrcorner f = 0$. Hence $\Gamma \lrcorner f^{hom,d_2} = 0$, and the claim follows.

  \item[(ii)] We have $\Ann(f)^{hom}\subseteq \Ann(f^{hom,d_2})$ by Part (i).
 Assume that $d\leq d_2$ and let $\Gamma=\alpha_0^d \Theta_0+\alpha_0^{d-1} \Theta_1+ \ldots  + \Theta_d$,
 where $ \Theta_i\in S^*_i$, be such that $\Gamma \lrcorner f^{hom,d_2} = 0$. We claim that
 $(\Gamma|_{\alpha_0=1}) \lrcorner f = 0$.
 
 By Equation \eqref{e:contracting} we have

\[
0  =\sum_{e=0}^{d_1}x_0^{[d_1+d_2-d-e]}\sum_{j=0}^{\min(d_1-e, d)}\Theta_j\lrcorner F_{e+j}.
\]
Since the exponents at $x_0$ are pairwise different, we have 
\[
\sum_{j=0}^{\min(d_1-e, d)} \Theta_j \lrcorner F_{e+j} = 0 \text{ for every } d_1\geq e \geq 0.
\]
This implies that $(\Gamma|_{\alpha_0=1}) \lrcorner f = 0.$

\item[(iii)] We start with the following

\noindent\textbf{Observation.} Assume that $k\geq 0$. For $\Gamma = \alpha_0^{d_1-1} \Theta_{1+k} + \alpha_0^{d_1-2}\Theta_{2+k}+ \ldots + \Theta_{d_1+k}$ we have
\[
\Gamma\lrcorner f^{hom,d_2} = 0 \Rightarrow \Gamma\in \Ann(f)^{hom}.
\]

Indeed, Equation \eqref{e:contracting} with $d_2=d_1-1, d=d_1+k$ becomes
\[
0=\Gamma\lrcorner f^{hom,d_2} =\sum_{e=0}^{d_1-k-1}x_0^{[d_1-k-e-1]}\sum_{j=0}^{d_1-e}\Theta_j\lrcorner F_{e+j}\text{.} 
\]
Since the exponents at $x_0$ are pairwise different, we have
\begin{equation}\label{e:almost_conclusion_2}
\sum_{j=0}^{d_1-e} \Theta_j \lrcorner F_{e+j} = 0 \text{ for every } d_1-k-1 \geq e \geq 0.
\end{equation}
For $d_1\geq e>d_1-k-1$ we have $\sum_{j=0}^{d_1-e} \Theta_j \lrcorner F_{e+j} =0$ since $\Theta_j = 0$ for $j<k+1$.
Together with Equation \eqref{e:almost_conclusion_2}, it implies that $\Gamma|_{\alpha_0=1}$ annihilates $f$ and thus $\Gamma\in \Ann(f)^{hom}$, as claimed.

We proceed to the proof of Part (iii). We claim that $\Ann(f^{hom,d_2})$ has at most one minimal homogeneous generator of degree $d_1$ modulo the generators of $(\Ann(f)^{hom})_{d_1}$. Indeed, by the above observation with $k=0$, any such generator is (up to a scalar) of the form $\alpha_0^{d_1}+\rho$, where $\alpha_0^{d_1}$ does not divide any monomial in $\rho$. Given two such generators, say $\alpha_0^{d_1}+\rho$ and $\alpha_0^{d_1}+\rho'$, we have $\alpha_0^{d_1}+\rho = (\alpha_0^{d_1}+\rho') + (\rho-\rho')$. From the above observation for $k=0$, it follows that $\rho-\rho'$ is in $(\Ann(f)^{hom})_{d_1}$, so the second new generator is not needed. Therefore, either 
\begin{align*}
  H(T^*/\Ann(f^{hom,d_2}),d_1) &= H(T^*/\Ann(f)^{hom},d_1) = r\text{, or} \\
  H(T^*/\Ann(f^{hom,d_2}),d_1) &= H(T^*/\Ann(f)^{hom},d_1)-1 = r - 1 \text{.}
\end{align*}

Now we assume $H(T^*/\Ann(f^{hom,d_2}), d_1) = r-1$. Then there exists a homogeneous
generator of $\Ann(f^{hom,d_2})$ of the form $\alpha_0^{d_1} + \rho$, where
$\alpha_0^{d_1}$ does not divide any monomial in $\rho$. It is enough to show
that for any $k\geq 0$, if $\Gamma = \alpha_0^{d_1-1} \Theta_{1+k} +
\alpha_0^{d_1-2}\Theta_{2+k}+ \ldots + \Theta_{d_1+k}$ annihilates $f^{hom, d_2}$, then $\Gamma
\in \Ann(f)^{hom}$. This is the observation from the beginning of the proof of
Part (iii).
\end{itemize}
\end{proof}


\begin{example}\label{e:divided_powers}
  Observe that Lemma \ref{l:ideals_agree_in_low_degree} works only when $f \in \kk_{dp}[x_0,\dots,x_n]$. We show what goes wrong when we use polynomial rings and the usual homogenization.
  Let $f = x_1^3 + x_2 \in \mathbb{C}[x_1,x_2]$ and  $G = x_1^3  + x_0^2 x_2
  \in \mathbb{C}[x_0,x_1,x_2]$ be its homogenization.
  Then
  \begin{equation*}
    \Ann(f)^{hom} = (\alpha_2^2, \alpha_1\alpha_2, \alpha_1^3 - 6\alpha_0^2\alpha_2)\text{,}
  \end{equation*}
  and
  \begin{equation*}
    \Ann(G) = (\alpha_0^3, \alpha_2^2, \alpha_0\alpha_1, \alpha_1\alpha_2, \alpha_1^3 - 3\alpha_0^2\alpha_2)\text{.}
  \end{equation*}
  The element $\alpha_1^3 - 6\alpha_0^2\alpha_2 \in \Ann(f)^{hom}$ does not annihilate $G$.
\end{example}

The following lemma is a generalization of Lemma \ref{l:ideals_agree_in_low_degree}. Here we use a subspace $W\subseteq S_{\leq d_1}$ instead of a polynomial $f\in S_{\leq d_1}$.

\begin{lemma}\label{l:subspace_case_annihilators} 
Let $W\subseteq S_{\leq d_1}$ be a linear subspace with $d_1\geq 1$ and fix a non-negative integer $d_2$.
We have:
\begin{itemize}
\item[(i)] $\Ann(W)^{hom}\subseteq \Ann(W^{hom, d_2})$,
\item[(ii)] $(\Ann(W)^{hom})_{\leq d_2} = \Ann(W^{hom, d_2})_{\leq d_2}$.
\end{itemize}
\end{lemma}
\begin{proof}
\hfill
\begin{itemize}
\item[(i)] Let $f\in W$. Then $\Ann(W)\subseteq \Ann(f)$. Therefore, 
\[
\Ann(W)^{hom}\subseteq \Ann(f)^{hom}\subseteq \Ann \left(\sum_{i=0}^{\deg f}  F_i x_0^{[d_2+d_1-i]}\right)
\]by Lemma \ref{l:ideals_agree_in_low_degree} (i). Varying $f$, this shows that 
\begin{equation*}
  \Ann(W)^{hom}\subseteq \bigcap_{H \in W^{hom, d_2}}\Ann(H) = \Ann(W^{hom, d_2})\text{.}
\end{equation*}
\item[(ii)] Let $\Theta\in \Ann(W^{hom, d_2})_{\leq d_2}$ be homogeneous and let $f\in W$. Then 
\[
\Theta \in \Ann\left(\sum_{i=0}^{\deg f}  F_i x_0^{[d_2+d_1-i]}\right)_{\leq d_2}.
\]
Since $d_2\leq d_2+d_1-\deg f$, it follows from Lemma~\ref{l:ideals_agree_in_low_degree}(ii) that $\Theta|_{\alpha_0 = 1}\in \Ann(f)$. We stress that when we apply Lemma \ref{l:ideals_agree_in_low_degree}(ii), we use $(\deg f, d_1+d_2-\deg f)$ instead of $(d_1,d_2)$. Since $f$ was arbitrary, we obtain 
\[
\Theta|_{\alpha_0 = 1} \in \bigcap_{f\in W} \Ann(f) = \Ann(W).
\]
Therefore, $\Theta \in \Ann(W)^{hom}$. 
\end{itemize}
\end{proof}

We recall some notation from Subsection \ref{ss:rank_and_border_rank} which will be used in the proof of the following lemma.
Let $\operatorname{Hilb}^{h_r}_{T^*}$ denote the multigraded Hilbert scheme associated 
with the polynomial ring $T^*$ (with the standard $\mathbb{Z}$-grading) and the function $h_r$, as defined in Definition \ref{d:standard_hilbert_function}. Let $\Slip_{r, \mathbb{P}T_1}$ be the closure in
$\operatorname{Hilb}_{T^*}^{h_r}$ of points corresponding to saturated ideals of $r$ points. Let $\Hilb_{r}(\PP^n)$ denote the Hilbert scheme of $r$ points on $\PP^n$ and $\mathcal{H}ilb^{sm}_{r}(\PP^n)$ denote the closure of the set of smooth schemes. 

\begin{lemma}\label{l:smoothable_then_br_of_hom_at_most_r}
 Let $d_1 \geq 1, d_2 \geq 0 $ be integers and $W \subseteq S_{\leq d_1}$ a linear subspace. Let $r= \dim_\kk S^*/\Ann(W)$.
If $S^*/\Ann(W)$ is smoothable, then the border rank of $W^{hom, d_2}$ is at most $r$.
\end{lemma}

\begin{proof}
Observe that $\Slip_{r, \PP T_1}$ surjects onto  $\mathcal{H}ilb^{sm}_{r}(\PP^n)$ 
  under the natural map 
  \[
  \operatorname{Hilb}_{T^*}^{h_r} \to \mathcal{H}ilb_{r}(\PP^n)
  \]
  given on closed points by $[I] \mapsto [\Proj T^*/I]$. Thus there is an ideal $[J]\in
  \Slip_{r, \PP T_1}$ with $J^{sat} = \Ann(W)^{hom}$ (we used Lemma \ref{l:homogenization_is_saturated}). Since $\Ann(W)^{hom}\subseteq
  \Ann(W^{hom, d_2})$  by Lemma \ref{l:subspace_case_annihilators}(i), we have $J\subseteq \Ann(W^{hom, d_2})$. Hence $[W^{hom, d_2}] \in \sigma_{r,\dim(W^{hom, d_2})}(\nu_{d}(\PP T_1))$ by the Border Apolarity Lemma \ref{p:border_apolarity}.
\end{proof}

\section{General results}\label{s:general_results}

In this section we apply the results of Section \ref{s:algebraic_results} to prove Theorems \ref{t:general_subspace} and \ref{t:general_polynomial}. These imply Theorems \ref{t:general_subspace_introduction} and \ref{t:general_polynomial_introduction}, respectively. 

We will use the following notation. Fix a positive integer $n$ and let
$S^*=\Bbbk[\alpha_1,\ldots,\alpha_n]\subseteq
T^*=\Bbbk[\alpha_0,\ldots,\alpha_n]$ be polynomial rings  over an algebraically closed field $\kk$ with graded dual rings
$S=\Bbbk_{dp}[x_1,\ldots,x_n]$ $\subseteq T=\Bbbk_{dp}[x_0,\ldots,x_n]$.
Recall the definitions of an $(s,n+1)$-standard Hilbert function given in  Definition \ref{d:standard_hilbert_function} and $W^{hom,d_2}$ from Definition~\ref{d:homogenization_of_space}.

We frequently use the following simple lemma.

\begin{lemma}\label{l:ideals_contain_each_other_if_agree_in_degree_d}
 Let $J, K \subseteq T^*$ be homogeneous ideals such that $J_s = K_s$ for a positive integer $s$. If $K$ is generated in degrees at most $s$, then $J_t \supseteq K_t$ for $t \geq s$. In particular:
 \begin{enumerate}[label=(\roman*)]
  \item if $J$ and $K$ have the same Hilbert polynomial then $J^{sat}=K^{sat}$,
  \item if the Hilbert polynomial of $T^*/K$ is zero, then the Hilbert polynomial of  $T^*/J$ is zero.
 \end{enumerate}

\end{lemma}
\begin{proof}
 We have $J_t \supseteq (J_s)_t = (K_s)_t = K_t$.
\end{proof}

\begin{theorem}[Subspace case]\label{t:general_subspace} 
Let $W\subseteq S_{\leq d_1}$ be a linear subspace and $r=\dim_\Bbbk S^*/\Ann(W)$.
We have the following:
\begin{itemize}
\item[(i)] The cactus rank  of $W^{hom, d_2}$ is not greater than $r$.
\item[(ii)] If $d_2\geq d_1$, then there is no homogeneous ideal $J\subseteq \Ann(W^{hom, d_2})$  such that $T^*/J$ has an $(r-1,n+1)$-standard Hilbert function. In particular, the border cactus rank  of $W^{hom, d_2}$  equals $r$.
\item[(iii)] If $d_2\geq d_1+1$, and $J\subseteq \Ann(W^{hom, d_2})$ is a homogeneous ideal such that $T^*/J$ has an $(r,n+1)$-standard Hilbert function, then $J^{sat} = \Ann(W)^{hom}$.
\end{itemize}
\end{theorem}

\begin{proof}
  \hfill
  \begin{itemize}
  \item[(i)]\label{i:1general_subspace} We have $\Ann(W)^{hom}\subseteq  \Ann(W^{hom, d_2})$ by Lemma \ref{l:subspace_case_annihilators}(i). Since the Hilbert polynomial of the algebra $T^*/\Ann(W)^{hom}$ is $r$ 
  by Corollary \ref{c:HF_of_extension_stabilizes}
  and the ideal $\Ann(W)^{hom}$ is saturated by Lemma \ref{l:homogenization_is_saturated}, the claim follows from the Cactus Apolarity Lemma \ref{p:cactus_apolarity}.
  
\item[(ii)]\label{i:2general_subspace} We have $H(T^*/\Ann(W)^{hom}, d_1) = r$ by Corollary \ref{c:HF_of_extension_stabilizes}. Therefore, by Lemma \ref{l:subspace_case_annihilators}(ii) we have 
\[
H(T^*/\Ann(W^{hom, d_2}), d_1) = r.
 \]
 Thus there exists no ideal $J\subseteq \Ann(W^{hom, d_2})$ such that $T^*/J$ has an $(r-1, n+1)$-standard Hilbert function.
  By the Weak Border Cactus Apolarity Lemma \ref{p:weak_border_cactus_apolarity_lemma} we get $\bcrr(W^{hom, d_2}) \geq r$, which together with Part (i) implies that $\bcrr(W^{hom, d_2}) = r$.

\item[(iii)]\label{i:3general_subspace} Assume that $J\subseteq \Ann(W^{hom, d_2})$ is such that $T^*/J$ has an $(r,n+1)$-standard Hilbert function. By Lemma \ref{l:subspace_case_annihilators}(ii) and Corollary \ref{c:HF_of_extension_stabilizes}
\[
H(T^*/\Ann(W^{hom, d_2}),d_2)=H(T^*/\Ann(W)^{hom}, d_2)=r.
\] 
In particular $J_{d_2} = (\Ann(W)^{hom})_{d_2}$. 
Since $\Ann(W)^{hom}$ is generated in degrees at most  $d_1+1 \leq d_2$ and the ideals $J$ and $\Ann(W)^{hom}$ have the same Hilbert polynomial, it follows from Lemma \ref{l:ideals_contain_each_other_if_agree_in_degree_d}(i) that 
$J^{sat} = (\Ann(W)^{hom})^{sat}=\Ann(W)^{hom}$. 
The last equality is true by Lemma \ref{l:homogenization_is_saturated}.
\end{itemize}
\end{proof}

Now we can state Theorem \ref{t:general_polynomial_introduction} in the following form, which includes the uniqueness statement hinted at in the introduction.

\begin{theorem}[Polynomial case]\label{t:general_polynomial}
 Let 
$f=F_{d_1}+F_{d_1-1}+\ldots +F_0 \in S=\kk_{dp}[x_1,...,x_n]$ be a degree $d_1 \geq 1$ polynomial, $r=\dim_\Bbbk S^*/\Ann(f)$. For a non-negative integer $d_2$, we have the following:
\begin{itemize}
\item[(i)] The cactus rank  of $f^{hom,d_2}$ is not greater than $r$.
\item[(ii)] If $d_2\geq d_1$, then there is no homogeneous ideal $J\subseteq \Ann(f^{hom,d_2})$ such that $T^*/J$ has an $(r-1,n+1)$-standard Hilbert function. 
Moreover, the same is true for $d_2=d_1-1$ if we assume further that  $F_{d_1}$ is not a power  of a linear form.

In particular, in both cases the border cactus rank of $f^{hom,d_2}$ equals $r$.
\item[(iii)] Assume that $F_{d_1}$ is not a power of a linear form. If  $d_2\geq d_1$ and $J\subseteq \Ann(f^{hom,d_2})$ is a homogeneous ideal  such that $T^*/J$ has an $(r,n+1)$-standard Hilbert function, then $J^{sat} = \Ann(f)^{hom}$. Moreover, the same is true for $d_2=d_1-1$ if we assume further that $r>2d_1$.
\end{itemize}

\end{theorem}

\begin{proof}
  \hfill
\begin{itemize}
\item[(i)] 
It follows directly from Theorem \ref{t:general_subspace}(i).
\item[(ii)] 
If $d_2\geq d_1$, 
then the claim follows from Theorem \ref{t:general_subspace}(ii). 

Suppose that $d_2=d_1-1$ and $F_{d_1}$ is not a power of a linear form. If $H(T^*/\Ann(f^{hom,d_2}), d_1) = r$, then there is no ideal $J\subseteq \Ann(f^{hom,d_2})$ such that $T^*/J$ has an $(r-1,n+1)$-standard Hilbert function. 
Suppose that
\begin{equation}\label{eq:notequal_r}
 H(T^*/\Ann(f^{hom,d_2}), d_1) \neq r.
\end{equation}
From Lemma \ref{l:extension_of_annihilator_generated_in_low_degrees} it follows that $\Ann(f)^{hom}$ is generated in degrees at most $ d_1$.
Then Equation \eqref{eq:notequal_r} and Lemma \ref{l:ideals_agree_in_low_degree}(iii) together imply that  $H(T^*/\Ann(f^{hom,d_2}), d_1)=r-1$ and $\Ann(f^{hom,d_2})$ has no minimal generator of degree greater than $d_1$. 
Let $J\subseteq \Ann(f^{hom,d_2})$ be a homogeneous ideal such that $T^*/J$ has an $(r-1,n+1)$-standard Hilbert function. Then we have $J_{d_1}=\Ann(f^{hom,d_2})_{d_1}$ since $H(T^*/\Ann(f^{hom,d_2}), d_1)=r-1=H(T^*/J, d_1)$.
Since the Hilbert polynomial of $T^*/\Ann(f^{hom,d_2})$ is $0$, it follows from Lemma \ref{l:ideals_contain_each_other_if_agree_in_degree_d}(ii) that the Hilbert polynomial of $T^*/J$ is $0$. This contradicts the fact that $T^*/J$ has an $(r-1, n+1)$-standard Hilbert function.

From the Weak Border Cactus Apolarity Lemma \ref{p:weak_border_cactus_apolarity_lemma},
it follows that $\operatorname{bcr}(f^{hom,d_2})\geq r$ and from Part~(i) we have an equality.
\item[(iii)] 
Suppose that $J\subseteq \Ann(f^{hom,d_2})$ is such that $T^*/J$ has an $(r,n+1)$-standard Hilbert function.
We will consider the following five cases:
\begin{enumerate}[label=(\Roman*)]
  \item $d_2\geq d_1$;
  \item $d_2=d_1-1$ and $H(T^*/\Ann(f^{hom, d_2}), d_1) = r$;
  \item $d_2=d_1-1$, $H(T^*/\Ann(f^{hom, d_2}),d_1) = r-1$ and $H(T^*/J, d_1) = r-1$;
  \item $d_2=d_1-1$, $H(T^*/\Ann(f^{hom, d_2}),d_1) = r-1$, $H(T^*/J, d_1) = r$ and $J_{d_1}=(\Ann(f)^{hom})_{d_1}$;
  \item $d_2=d_1-1$, $H(T^*/\Ann(f^{hom, d_2}),d_1) = r-1$, $H(T^*/J, d_1) = r$ and $J_{d_1}\neq (\Ann(f)^{hom})_{d_1}$.
\end{enumerate}
We explain that these are the only possible cases. Suppose that $d_2=d_1-1$ and 
\[
H(T^*/\Ann(f^{hom, d_2}), d_1) \neq r. 
\]
Then 
\[
 H(T^*/\Ann(f^{hom, d_2}), d_1) = r-1
\]
by Lemma \ref{l:ideals_agree_in_low_degree}(iii).
It suffices to show that if $H(T^*/\Ann(f^{hom,d_2}),d_1) = r - 1$, then $H(T^*/J, d_1) \in \{r-1, r\}$.
This holds since $T^*/J$ has an $(r,n+1)$-standard Hilbert function and $J\subseteq \Ann(f^{hom,d_2})$.

We prove that $J^{sat}=\Ann(f)^{hom}$ in each case.
\begin{enumerate}[label=(\Roman*)]
\item  By Lemma~\ref{l:subspace_case_annihilators}(ii) and Corollary~\ref{c:HF_of_extension_stabilizes}
\[
H(T^*/\Ann(f^{hom,d_2}),d_2)=H(T^*/\Ann(f)^{hom}, d_2)=r.
\] 
In particular $J_{d_2} = (\Ann(f)^{hom})_{d_2}$. 
The ideal $\Ann(f)^{hom}$ is generated in degrees at most  $d_1 \leq d_2$, by Lemma~\ref{l:extension_of_annihilator_generated_in_low_degrees}.  The ideals $J$ and $\Ann(f)^{hom}$ have the same Hilbert polynomial, so by Lemma~\ref{l:ideals_contain_each_other_if_agree_in_degree_d}(i), we have $J^{sat} = (\Ann(f)^{hom})^{sat}=\Ann(f)^{hom}$. 
The last equality is true by Lemma~\ref{l:homogenization_is_saturated}.

\item We have $J_{d_1} = \Ann(f^{hom,d_2})_{d_1} = (\Ann(f)^{hom})_{d_1}$. 
The ideal $\Ann(f)^{hom}$ is generated in degrees at most $d_1$ by Lemma \ref{l:extension_of_annihilator_generated_in_low_degrees}.
The ideals $J$ and $\Ann(f)^{hom}$ have the same Hilbert polynomial, so by Lemma \ref{l:ideals_contain_each_other_if_agree_in_degree_d}(i), we have $J^{sat} = (\Ann(f)^{hom})^{sat}=\Ann(f)^{hom}$. 
 The last equality is true by Lemma~\ref{l:homogenization_is_saturated}.

\item We have $J_{d_1} = \Ann(f^{hom,d_2})_{d_1}$ and the ideal $\Ann(f^{hom,d_2})$ is generated in degrees at most $d_1$ by Lemmas~\ref{l:ideals_agree_in_low_degree}(iii) and
\ref{l:extension_of_annihilator_generated_in_low_degrees}.
The Hilbert polynomial of $T^*/\Ann(f^{hom,d_2})$ is zero, so by Lemma~\ref{l:ideals_contain_each_other_if_agree_in_degree_d}(ii) the Hilbert polynomial of $T^*/J$ is zero. This contradicts the assumption that $T^*/J$ has an $(r,n+1)$-standard Hilbert function.

\item Proof is as in (II).

\item The ideal $J$ has a generator of the form $\alpha_0^{d_1}+\rho$, where $\rho\in
T^*_{d_1}$ has degree smaller than $d_1$ with respect to $\alpha_0$ (again by Lemma \ref{l:ideals_agree_in_low_degree}(iii)). Since
$\operatorname{codim}_{\Ann(f^{hom,d_2})_{d_1}}J_{d_1} = 1$, we have
\[
 \operatorname{codim}_{(\Ann(f^{hom,d_2})^c)_{d_1}}(J^c)_{d_1} \leq 1.
\]
Here $K^c$ denotes
$K\cap S^*$ for any ideal $K\subseteq T^*$. We shall consider $I =
\Ann(F_{d_1})$. We have $I_{d_1}\subseteq (\Ann(f^{hom,d_2})^c)_{d_1}$ and $H(S^*/I,
d_1)=1$. Therefore, we have 
\[
H(S^*/J^c, d_1)\leq H(S^*/\Ann(f^{hom,d_2})^c, d_1) + 1\leq H(S^*/I, d_1) + 1 = 2\text{.}
\] 
Since $d_1\geq 2$, it follows from the Macaulay's bound (\cite{BH98},
Theorem 4.2.10) that  for $d\geq d_1$ we have $H(S^*/J^c, d)\leq 2$. Hence
\[
H(T^*/J, d) \leq H(S^*/J^c, d)+\ldots +H(S^*/J^c, d-(d_1-1)) \leq 2d_1 < r
\]
for $d\geq 2d_1-1$.  We used here Lemma \ref{l:hilbert_function_and_contraction}. This gives a contradiction since the Hilbert polynomial of $T^*/J$ is equal to $r$.

\end{enumerate}
\end{itemize}
\end{proof}

The following examples show that the assumptions of Theorem \ref{t:general_polynomial} are in general as sharp as possible.

\begin{example}\label{ex:example_1}
Let $S=\Bbbk_{dp}[x_1,x_2]$, $f=x_1^{[2]}+x_1x_2$ and assume $d_2=d_1-2=0$. Then $r=4$ and $\Ann(f^{hom,0}) = (\alpha_0, \alpha_1^2-\alpha_1\alpha_2,\alpha_2^2)$. Consider the ideal $J=(\alpha_0^2, \alpha_0\alpha_1, \alpha_1^2-\alpha_1\alpha_2)$. Then $\kk[\alpha_0, \alpha_1, \alpha_2]/J$ has Hilbert function $h_3$. Therefore, the assumption $d_2\geq d_1-1$ in Theorem \ref{t:general_polynomial} (ii) cannot be weakened in general.
\end{example}

\begin{example}
As in Example \ref{ex:example_1}, let $S=\Bbbk_{dp}[x_1,x_2]$ and $f=x_1^{[2]}+x_1x_2$. Then  $r=4 = 2d_1$. If $d_2=1$, then $\Ann(f^{hom,1}) = (\alpha_0^2, \alpha_1^2-\alpha_1\alpha_2,\alpha_2^2)$. The ideal $J=(\alpha_0^2, \alpha_1^2-\alpha_1\alpha_2)$ is saturated 
and $\kk[\alpha_0, \alpha_1, \alpha_2]/J$ has Hilbert function $h_4$. However, $J$ does not contain $\alpha_2^2 \in \Ann(f)^{hom}$. Therefore, the assumption $r>2d_1$ in Theorem \ref{t:general_polynomial} (iii) cannot be skipped. 
\end{example}

\begin{example}
Let $S=\Bbbk_{dp}[x_1,x_2,x_3]$ and $f=x_1x_2x_3$. Then $r=8 > 6 = 2d_1$. If $d_2=d_1-2=1$, then $\Ann(f^{hom,1})=(\alpha_0^2, \alpha_1^2, \alpha_2^2, \alpha_3^2)$. Consider the ideal $J=(\alpha_0^3, \alpha_0^2\alpha_1, \alpha_1^2,\alpha_0^2\alpha_2, \alpha_2^2,\alpha_0^2\alpha_3)$. Then
$\kk[\alpha_0, \alpha_1, \alpha_2, \alpha_3]/J$ has Hilbert function $h_8$ and $J^{sat}= (\alpha_0^2, \alpha_1^2, \alpha_2^2) \neq \Ann(f)^{hom}$. Therefore, the assumption $d_2\geq d_1-1$ in Theorem \ref{t:general_polynomial} (iii) cannot be weakened in general.
\end{example}

\section[\texorpdfstring{14-th cactus variety of $d$-th Veronese embedding of $\mathbb{P}^n$}{14-th cactus variety of 
d-th Veronese embedding of projective space of dimension n}]{14-th cactus variety of $d$-th Veronese embedding of $\mathbb{P}^n$}\label{s:14thsecant}

In this section we assume that $d \geq 5$ and $n\geq 6$ are integers. We show that the cactus variety $\kappa_{14}(\nu_d(\mathbb{P}^n))$ has two irreducible components, one of which is the secant variety $\sigma_{14}(\nu_d(\PP^n))$ and we describe the other one (see Theorem \ref{t:segre-veronese_map_general_n}).
Furthermore for $n\geq 6$ and $d \geq 6$ we present an algorithm (Theorem \ref{t:algorithm}) for deciding whether 
$[G] \in \kappa_{14}(\nu_{d}(\mathbb{P}^n))$ is in $ \sigma_{14}(\nu_{d}(\mathbb{P}^n))$.

Since the results of this section depend on the paper \cite{Jel16}, 
in which the author works over the field of complex numbers, in this section we will assume that $\Bbbk = \CC$. 
In that case, the graded dual ring of a polynomial ring is isomorphic to a polynomial ring.

Let $\mathcal{H}ilb^{Gor}_{r}(X)$, where $X=\mathbb{A}^n$ or $\mathbb{P}^n$, denote the open subset of the Hilbert scheme of $r$ points on $X$ consisting of Gorenstein
subschemes. Since we shall use the results of Casnati, Jelisiejew, and Notari from \cite{CJN15} in this section, we give a brief summary.
\begin{theorem}[Casnati, Jelisiejew, Notari, \cite{CJN15}]\label{t:1661_summary}
  We have the following:
  \begin{enumerate}[label=(\roman*)]
    \item the scheme $\mathcal{H}ilb_{r}^{Gor}(\AA^n)$ is irreducible for $r < 14$ and any $n\in \mathbb{N}$,
    \item the scheme $\mathcal{H}ilb_{14}^{Gor}(\AA^n)$ is reducible if and only if $n\geq 6$,
    \item if the scheme $\mathcal{H}ilb_{14}^{Gor}(\AA^n)$ is reducible, it has
      two irreducible components: $\mathcal{H}ilb_{14}^{Gor,sm}(\AA^n)$, the
      closure of the set of smooth schemes, and $\mathcal{H}^n_{1661,af}$, the
      closure of the set of local algebras with local Hilbert function
      $(1,6,6,1)$.
  \end{enumerate}
\end{theorem}
\begin{proof}
  See \cite[Thm. A and B]{CJN15} for Parts (i) and (ii). Part (iii) follows from \cite[Thm. 6.17 and Lem. 6.19]{CJN15}. For a precise proof, see \cite[p. 1567]{CJN15}. 
\end{proof}
In particular, it follows from Theorem \ref{t:1661_summary} that $\mathcal{H}ilb_r^{Gor}(\PP^n)$ is irreducible for $r < 14$. Therefore, in that case,
$\kappa_r(\nu_d(\mathbb{P}^n)) = \sigma_r(\nu_d(\mathbb{P}^n))$.
Indeed, we have
\begin{equation}\label{eq:gorenstein_give_catus}
\kappa_r(\nu_d(\mathbb{P}^n)) = \overline{\bigcup \{\langle \nu_d(R) \rangle \mid [R]\in\mathcal{H}ilb_r^{Gor}(\PP^n)\}}
\end{equation}
by \cite[Prop.~2.2]{BB14}. Therefore,  irreducibility of
$\mathcal{H}ilb^{Gor}_r(\PP^n)$ implies $\kappa_r(\nu_d(\PP^n)) =
\sigma_r(\nu_d(\PP^n))$. If $n < 6$, then $\mathcal{H}ilb^{Gor}_{14}(\PP^n)$ is irreducible. Therefore, the cactus variety
$\kappa_{14}(\nu_d(\PP^n))$ is irreducible for $n < 6$. Note that a description of the cactus variety,
similar to the one given by Equation \eqref{eq:gorenstein_give_catus}, works
over an arbitrary field (see \cite[Cor.~6.20]{BJ17}).

Part (iii) of Theorem \ref{t:1661_summary} is the reason why in the next subsection we analyze
algebras with local Hilbert function $(1,6,6,1)$. It
follows from the theory of Macaulay's inverse systems (\cite[Thm.
21.6]{Eis95}), that every such algebra is the apolar algebra of a cubic
polynomial.

\subsection[\texorpdfstring{The set of cubics with Hilbert function $(1,6,6,1)$}{The set of cubics with Hilbert function (1,6,6,1)}]{The set of cubics with Hilbert function $(1,6,6,1)$}
In Lemma
\ref{l:equivalent_description_of_1661}, we give a useful characterization of
cubics $f$ such that the Hilbert function of $\Apolar(f)$ is $(1,6,6,1)$. This
is inspired by \cite[Ex. 8]{BJMR17}. Then we establish Lemma
\ref{l:D_i_is_dense_and_of_dimension_13n+5} about topological properties of the set of such cubics.

In this subsection $S^* = \CC[\alpha_1,\dots,\alpha_n]$, and $S =
\CC[x_1,\dots,x_n]$ is its graded dual. We assume that $n\geq 6$.  Given $f \in
S$, we denote by $F_j$ its homogeneous part of degree $j$. 

 \begin{lemma}\label{l:calculation}
   Let $W \subseteq S$ be a linear subspace. Then 
   \begin{equation*}
     H(\Apolar(W), k)  =  \operatorname{codim}_{S^*_k} E_k\text{,}
   \end{equation*}
   where $E_k = \{\theta_k \in S^*_k \mid$ there exists $\theta_{\geq k+1} \in
   S^*_{\geq k+1}$ such that $(\theta_k + \theta_{\geq k+1})\lrcorner W = 0\}$.
 \end{lemma}
 \begin{proof}
    Let $\overline{\mathfrak{m}}$ be the maximal ideal of $\Apolar(W)$. Then
  \begin{align*}
    H(\Apolar(W),k) &= \dim_\CC \overline{\mathfrak{m}}^k/\overline{\mathfrak{m}}^{k+1} \\
    &=\operatorname{codim}_{S^*_{\geq k}} \Ann(W)\cap S^*_{\geq k} - \operatorname{codim}_{S^*_{\geq k+1}} \Ann(W)\cap S^*_{\geq k+1}\\
    &= \operatorname{codim}_{S^*_{\geq k}}S^*_{\geq k+1} - \operatorname{codim}_{\Ann(W)\cap S^*_{\geq k}} \Ann(W)\cap S^*_{\geq k+1} \\
    &= \dim_\CC S^*_k -\dim_\CC \frac{\Ann(W)\cap S^*_{\geq k}}{\Ann(W)\cap S^*_{\geq k+1}} \\
    &= \dim_\CC S^*_k - \dim_\CC E_k.
\end{align*}
\end{proof}
\begin{lemma}\label{l:equivalent_description_of_1661}
 For $[f]\in \PP S_{\leq 3}$ the following are equivalent:
\begin{enumerate}[label = {(\alph*)}]
\item $\Apolar(f)$ has Hilbert function $(1,6,6,1)$,
\item there exists $[U] \in \Gr(6,S_1)$ such that $F_3 \in \Sym^3 U$, $F_2 \in U \cdot S_1$ and $H(\Apolar (F_3),1)=6$.
\end{enumerate}
\end{lemma}
\begin{proof}
   By Iarrobino's symmetric decomposition (see \cite[Thm.~2.3~and the following remarks]{CJN15}), the algebra $\Apolar(f)$ has Hilbert
   function $(1,c+e,c,1)$, where $(1,c,c,1)$ is the Hilbert function of
   $\Apolar(F_3)$. From Lemma \ref{l:calculation}, we know that $c+e =
   \operatorname{codim}_{S^*_1} E_1$, where 
   $$E_1 = \{\theta_1 \in S^*_1 \mid \text{ there
     exists } \theta_{\geq 2} \in S^*_{\geq 2} \text{ such that } (\theta_1 +
     \theta_{\geq 2})\lrcorner f = 0\}.$$
     We use the following computation
     \begin{align}\label{eq:theta_lrcorner}
     \begin{split}
      (\theta_3 + \theta_2 + \theta_1) \lrcorner (F_3 + F_2 + F_1 + F_0)
      =( \theta_1 \lrcorner F_3) + ( \theta_1 \lrcorner F_2 +  \theta_2 \lrcorner F_3) + ( \theta_1 \lrcorner F_1 + \theta_2 \lrcorner F_2  +  \theta_3 \lrcorner F_3).
     \end{split}
     \end{align}
     
     Assume that $\Apolar(f)$ has Hilbert function $(1,6,6,1)$. We will show that condition (b) is satisfied. 
     Let $U$ be $S^*_2 \lrcorner F_3$, which is 6 dimensional, since the Hilbert function of $\Apolar(F_3)$ is $(1,6,6,1)$ by the above discussion. It is enough to show that $F_2 \in U \cdot S_1$. 
     Assume that this does not hold. Up to a linear change of variables we can assume that $U=\langle x_1,x_2,\dots,x_6 \rangle$. Let $V=\langle x_7, x_8, \dots , x_n \rangle$. By classification of quadratic forms over $\CC$ we may assume that $F_2 = x_n^2 + H+K$, where $H\in \Sym^2(\langle x_7, x_8, \ldots, x_{n-1}\rangle)$ and $K\in S_1 \cdot U$. Then $\alpha_n\lrcorner F_2 \notin U$ and hence $\alpha_n \notin E_1$ by Equation \eqref{eq:theta_lrcorner}. 
     
     We claim that $\dim_\CC E_1 \leq n-7$. It suffices to show that the classes of $\alpha_1, \alpha_2, \ldots, \alpha_6,  \alpha_n$ are linearly independent in the vector space $S^*_1/E_1$.
%
%
    Let $\omega = a_1\alpha_1+a_2\alpha_2+\ldots + a_6\alpha_6$ for some $a_1, \ldots, a_6\in \CC$ and assume that there is $b\in \CC$ and $\theta_{\geq 2} \in S^*_{\geq 2}$ such that $(\omega + b\alpha_n + \theta_{\geq 2}) \lrcorner f = 0$.
    Then by Equation \eqref{eq:theta_lrcorner} we get $(\omega +  b\alpha_n )\lrcorner F_3 = 0$. Since $U = S^*_2 \lrcorner F_3 = \langle x_1, x_2,\ldots, x_6 \rangle$, it follows that $F_3$ is a polynomial in $x_1, \ldots, x_6$. Therefore, $b\alpha_n \lrcorner F_3 = 0$ and as a consequence $\omega \lrcorner F_3 = 0$. We know that $\Apolar(F_3)$ has Hilbert function $(1,6,6,1)$, thus $\omega = 0$. As a result, $b\alpha_n\in E_1$, which shows that $b=0$.
     
    Having established that $\dim_\CC E_1 \leq n-7$, we get from  Lemma \ref{l:calculation} that $H (\Apolar (f),1)\geq 7$. This contradicts the assumption that $H(\Apolar(f), 1) =6$. 
     
     For the other direction, assume that (b) holds, we will show that $\Apolar(f)$ has Hilbert function $(1,6,6,1)$. It is enough to show that $\operatorname{codim}_{S_1^*} E_1 = 6$.      
     By assumption $\dim_\CC \Ann(F_3)_1 = n-6$, so it suffices to show that $E_1=\Ann(F_3)_1$. Assume that $\theta=\theta_3 + \theta_2 + \theta_1 \in \Ann (f)$. Then it follows from Equation \eqref{eq:theta_lrcorner} that $\theta_1 \in  \Ann(F_3)_1$. Thus $E_1\subseteq \Ann(F_3)_1$.
     Let us take $\theta_1 \in \Ann(F_3)_1$. From the assumption $F_2= \sum_i u_i h_i$, where $u_i \in U, h_i \in S_1$. Therefore, 
     \[
        \theta_1 \lrcorner F_2 = \sum_i u_i (\theta_1 \lrcorner h_i) \in U.     
     \]
     Since $(-)\lrcorner F_3: S_2^* \to U$ is surjective, there exists $\theta_2 \in S_2^*$ such that $\theta_2 \lrcorner F_3 = -\theta_1 \lrcorner F_2$. By Equation \eqref{eq:theta_lrcorner} it is enough to observe that there exists $\theta_3 \in S_3^*$ such that $\theta_3 \lrcorner F_3 = - (\theta_1 \lrcorner F_1 + \theta_2 \lrcorner F_2)$.

\end{proof}

\begin{lemma}\label{l:D_i_is_dense_and_of_dimension_13n+5}
The following subset 
\begin{align*}
A=\{[f]\in \PP S_{\leq 3} \mid& \Apolar(f) \text{ has Hilbert function }(1,6,6,1)\}
\end{align*}
is irreducible, of dimension $13n+5$, and locally closed.
Moreover the set
\begin{align*}
B=\{[f]\in A \mid& [\Spec \Apolar(f)] \notin \Hilb_{14}^{Gor,sm}(\mathbb{A}^n)\}
\end{align*}
is dense in $A$.
\end{lemma}
\begin{proof}
Consider 
\[
\mathcal{A}= \{([U],[f]) \in \Gr(6, S_{1}) \times \PP S_{\leq3}  \mid  [f] \in \PP(\Sym^3 U\oplus (S_{1}\cdot U) \oplus S_{\leq 1})\}
\]
and 
\[
\mathcal{A}^0= \{([U],[f]) \in \mathcal{A} \mid H(\Apolar(F_3),1)=6\}\text{.}
\]
We have a pullback diagram
\begin{center}
\begin{tikzcd}
\mathcal{A} \arrow[r] \arrow[d] & \operatorname{Fl}(1,7n+42 , S_{\leq 3}) \arrow[d] \\
\operatorname{Gr}(6, S_1) \arrow[r] & \Gr(7n+42, S_{\leq 3})
\end{tikzcd}
\end{center}
where $\operatorname{Fl}(1,7n+42 , S_{\leq 3})$ is the flag variety parametrizing flags of subspaces  $M\subseteq N \subseteq S_{\leq 3}$ with $\dim_\CC M = 1$, $\dim_\CC N = 7n + 42$ and  the lower horizontal map sends $[U]$ to $[\Sym^3 U\oplus (S_{1}\cdot U) \oplus S_{\leq 1}]$.

The varieties $\mathcal{A}$ and $\Gr(6, S_1)$ are projective. Moreover, the left vertical map is surjective and its fibers are irreducible varieties
isomorphic to $\PP^{7n+41}$. Since $\PP^{7n+41}$ is irreducible, it follows from \cite[Thm.~1.25 and 1.26]{Sha13} that $\mathcal{A}$ is irreducible and of dimension $6(n-6)+7n+41 = 13n+5$.

We will show that $\mathcal{A}^0$ is open in $\mathcal{A}$. Consider the subset
\[
\mathcal{B}= \{([U],[f]) \in \Gr(6, S_1) \times \PP{S_{\leq 3}} \mid  H(\Apolar(F_3),1) \geq 6 \}\text{.}
\]
Observe that $\mathcal{A}^0= \mathcal{A} \cap \mathcal{B}$. It is enough to show that $\mathcal{B}$ is open in $\Gr(6, S_1) \times \PP{S_{\leq 3}}$. 
Let
\[
 \mathcal{C} = \{ [f] \in \PP{S_{\leq 3}} \mid H(\Apolar(F_3),1) \geq 6 \}.
\]
It suffices to show that $\mathcal{C}$ is open in $\PP{S_{\leq 3}}$, which holds since its complement is given by catalecticant minors. We have established that $\mathcal{A}^0= \mathcal{A} \cap \mathcal{B}$ is open in $\mathcal{A}$.

By Lemma \ref{l:equivalent_description_of_1661} we have $A = \pi_2(\mathcal{A}^0)$, where $\pi_2 \colon \Gr(6, S_{1}) \times \PP S_{\leq3}  \to \PP S_{\leq 3}$ is the projection.
Since $\pi_2|_{\mathcal{A}^0}\colon \mathcal{A}^0 \to A$ has a finite fiber over every point, it follows from \cite[Thm.~11.4.1]{RV2017} that $A$ is irreducible and of dimension $13n+5$. 

We know that $A=\pi_2(\mathcal{A}^0)= \pi_2 (\mathcal{A}) \cap \mathcal{C}$, which is locally closed since $\pi_2(\mathcal{A})$ is closed and $\mathcal{C}$ is open. Therefore we have a morphism $\mu : A \to \Hilb_{14}^{Gor}(\AA^n)$ given on closed points by $[f] \mapsto [\Spec S^*/\Ann(f)]$, see Theorem \ref{t:morphism}. 

By Theorem \ref{t:1661_summary}, the scheme $\Hilb_{14}^{Gor}(\mathbb{\AA}^n)$
has two irreducible components
$\Hilb_{14}^{Gor,sm}(\mathbb{\AA}^n)$ and $\mathcal{H}^n_{1661,af}$.
We obtain $B = \mu^{-1}(\mathcal{H}^n_{1661,af} \setminus \Hilb_{14}^{Gor,sm} (\AA^n))$, so it is open in $A$. Since $B$ is non-empty 
(see \cite[Rmk. 3.7]{Jel16}, and \cite[Thm. 3.16, Prop. 5.6]{BJ17})
and $A$ is irreducible, it follows that $B$ is dense in $A$.
\end{proof}
\subsection{Proofs of the main theorems}
We will consider the polynomial ring $T^* = \CC [\alpha_0,\alpha_1,\ldots,\alpha_n]$, and its graded dual $T = \CC[x_0,x_1,\dots,x_n]$,  where $n \geq 6$. Since we assume $\kk = \CC$, the graded dual ring $T$ is isomorphic to a polynomial ring.
Given $f \in T$, we denote by $F_j$ its homogeneous part of degree $j$. Recall
Definition \ref{d:prime_operator}.

Our goal is to characterize for $d \geq 5$ and $ n \geq 6$ the closure of the set-theoretic difference between the cactus variety $\kappa_{14}(\nu_d(\PP T_1))$ and the secant variety $\sigma_{14}(\nu_d(\PP T_1))$. 
For $n=6$ and $d\geq 5$ this closure consists of points $[G]\in \PP T_d$ with $G$ divisible by $(d-3)$-rd power of a linear form. However for $n > 6$ the situation is more complicated.

For $d \geq 3$ we will define a subset $\eta_{14}(\nu_d(\PP^n))$ of the cactus variety $\kappa_{14}(\nu_d(\PP^n))$. Later, in Theorem \ref{t:segre-veronese_map_general_n}, it will be shown that for $d\geq 5$
\[
\kappa_{14}(\nu_d(\PP^n)) = \sigma_{14}(\nu_d(\PP^n)) \cup \eta_{14}(\nu_d(\PP^n))
\]
is the decomposition into irreducible components.

Consider the following rational map $\varphi$, which assigns to a scheme
$R$ its projective linear span $\langle v_d(R) \rangle$
\[\begin{tikzcd}
    \varphi: \Hilb_{14}^{Gor}(\mathbb{P}^n) \ar[dashed]{r} & \Gr(14, \Sym^{d} \mathbb{C}^{n+1})\text{.}
\end{tikzcd}\]
Let $U \subseteq \Hilb_{14}^{Gor}(\mathbb{P}^n)$ be a dense open subset on which $\varphi$ is regular.

Consider the projectivized universal bundle $\mathbb{P}\mathcal{S}$ over $\Gr(14, \Sym^{d}\mathbb{C}^{n+1})$, given as a set by 
\begin{equation*}
  \mathbb{P}\mathcal{S} = \{ ([P],[p]) \in \Gr(14, \Sym^{d}\mathbb{C}^{n+1})\times \mathbb{P}(\Sym^{d}\mathbb{C}^{n+1}) \mid p \in P \}\text{,} 
\end{equation*}
together with the inclusion
$i: \mathbb{P}\mathcal{S} \hookrightarrow
\Gr(14,\Sym^{d}\mathbb{C}^{n+1}) \times
\mathbb{P}(\Sym^{d}\mathbb{C}^{n+1})$.
We pull the commutative diagram
\[\begin{tikzcd}
    \mathbb{P}\mathcal{S} \ar[hookrightarrow]{rr}{i}\ar{dr}{\pi} & &  \Gr(14,\Sym^{d}\mathbb{C}^{n+1}) \times \mathbb{P}(\Sym^{d}\mathbb{C}^{n+1}) \ar{dl}{\pr_1} \\
    & \Gr(14,\Sym^{d}\mathbb{C}^{n+1})
\end{tikzcd}\]
back along $\varphi$ to $U$, getting the commutative diagram
\[\begin{tikzcd}
    \varphi^* (\mathbb{P}\mathcal{S}) \ar[hookrightarrow]{rr}{\varphi^* i}\ar{dr}{\varphi^*\pi} & &  U \times \mathbb{P}(\Sym^{d}\mathbb{C}^{n+1}) \ar{dl}{\pr_1} \\
    & U\text{.}
\end{tikzcd}\]
Let $Y$ be the closure of $\varphi^*(\mathbb{P}\mathcal{S})$ inside
$\Hilb_{14}^{Gor}(\mathbb{P}^n)\times \PP (\Sym^{d} \mathbb{C}^{n+1})$. The
scheme $Y$ has two irreducible components, $Y_1$ and $Y_2$, corresponding to
two irreducible components of $\Hilb_{14}^{Gor}(\mathbb{P}^n)$, the schemes
$\Hilb_{14}^{Gor,sm}(\mathbb{P}^n)$ and $\mathcal{H}_{1661}$, respectively.
For the description of irreducible components of $\Hilb_{14}^{Gor}(\mathbb{P}^n)$, see Theorem \ref{t:1661_summary}.

Then
\begin{align}
\kappa_{14}(\nu_d(\PP^n)) &= \pr_2(Y) \label{eq:kappa_is_projection}, \\  
  \sigma_{14}(\nu_{d}(\mathbb{P}^n)) &= \pr_2(Y_1)\text{, and we define} \label{eq:sigma_is_projection}\\
  \eta_{14}(\nu_{d}(\mathbb{P}^n)) &:= \pr_2(Y_2)\label{eq:eta_is_projection}\text{.}
\end{align}

In Proposition \ref{p:dim_eta}, we bound from above the dimension of the irreducible subset $\eta_{14}(\nu_d(\PP^n))$ by $14n+5$.
Later, in Theorem \ref{t:segre-veronese_map_general_n}, we will identify a $(14n+5)$-dimensional subset of $\kappa_{14}(\nu_d(\PP^n))\setminus \sigma_{14}(\nu_d(\PP^n))$. It will allow us to conclude that the closure of this subset is $\eta_{14}(\nu_d(\PP^n))$.

\begin{lemma}\label{l:dim_H1661}
For $n\geq 6$ the component $\mathcal{H}_{1661} $ has dimension $14n-8$. Let $[R]\in \mathcal{H}_{1661}\subseteq \Hilb^{Gor}_{14}(\PP^n)$ be a non-smoothable subscheme. 
Then the dimension of the tangent space $\dim_\CC T_{[R]}\Hilb^{Gor}_{14} (\PP^n)$ equals $14n-8$. 
\end{lemma}
\begin{proof}
For $m \geq 6$ we use the notation $\mathcal{H}_{1661}^m$ for the non-smoothable component of $\Hilb_{14}^{Gor}(\PP^m)$. 
Let $\AA^m \subseteq \PP^m$ be the complement of a hyperplane, then $\Hilb^{Gor}_{14} (\AA^m)$ is an open subset of  $\Hilb^{Gor}_{14} (\PP^m)$. We denote $\mathcal{H}_{1661}^{m} \cap \Hilb^{Gor}_{14} (\AA^m)$ by $\mathcal{H}_{1661,af}^{m}$, as in Theorem \ref{t:1661_summary}.
Let $Z^m:=\{[R]\in \mathcal{H}_{1661,af}^{m} \mid \operatorname{Supp}(R)=0 \}$.
Observe that $\dim Z^m = \dim \mathcal{H}_{1661,af}^{m} - m$, since $\mathcal{H}_{1661,af}^{m}$ is a trivial $\AA^m$-bundle over $Z^m$.
By \cite[Prop. A.4]{BJJMM19} we have $\dim Z^{n} = 13n + \dim Z^6 - 78$. Thus  $\dim Z^{n} = 13n-8$, since $\dim Z^6= \dim \mathcal{H}_{1661,af}^6 -\dim \AA^6=76 -6=70$, see \cite[Thm. 1.1.~parts~2., 3.]{Jel16}. It follows that $\dim \mathcal{H}_{1661}^n = \dim \mathcal{H}_{1661,af}^n = 14n - 8$.

Let $R' \subseteq \PP^6$ be a subscheme abstractly isomorphic with $R$.  From \cite[Lem. 2.3]{CN09} we have \[
\dim_\CC T_{[R]}\Hilb^{Gor}_{14} (\PP^n) = 14n + T_{[R']}\Hilb^{Gor}_{14} (\PP^6) -84.
\]
From \cite[Thm. 1.1]{BJ17} $R'$ is non-smoothable, hence $\dim T_{[R']}\Hilb^{Gor}_{14} (\PP^6) = 76$  by \cite[Claim 3]{Jel16}.
\end{proof}

\begin{proposition}\label{p:dim_eta}
The dimension of $\eta_{14}(\nu_{d} (\PP^n))$ is at most $14n+5$.
\end{proposition}
\begin{proof}
We have the following commutative diagram

\begin{center}
\begin{tikzcd}[column sep = small, row sep = small]
\PP (\Sym^{d} \CC^{n+1}) \supseteq \sigma \cup \eta & Y_1\cup Y_2 \arrow[l] \arrow{d}{\chi} \arrow[r, dashed] & \PP \mathcal{S} \arrow[d] \\
\mathcal{H}ilb_{14}^{Gor} (\PP^n) \arrow[r, equal] & \mathcal{H}ilb_{14}^{Gor,sm} (\PP^n) \cup \mathcal{H}_{1661} \arrow[r, dashed] &  \Gr (14, \Sym^{d} \CC^{n+1})
\end{tikzcd}
\end{center}

\noindent where $\sigma$ and $\eta$ denote $\sigma_{14}( \nu_{d} (\PP^n))$ and $\eta_{14} ( \nu_{d} (\PP^n))$, respectively, and $\chi:Y_1\cup Y_2 \to
\Hilb_{14}^{Gor}(\mathbb{P}^n)$ is the projection. 
Then $\dim \eta_{14}(\nu_{d} (\PP^n)) \leq \dim (Y_2)
= \dim \mathcal{H}_{1661} + 13$, where $13$ is the dimension of the general fiber of the map $\chi|_{Y_2}: Y_2
  \to \mathcal{H}_{1661}$. 
It follows from Lemma \ref{l:dim_H1661}, that $\dim \mathcal{H}_{1661}=14n-8$ and therefore $\dim \eta_{14}(\nu_{d} (\PP^n)) \leq 14n+5$.
\end{proof}

In the rest of the section we use the notation $f^{\btd d}$ as introduced in Definition \ref{d:prime_operator}. 

\begin{proposition}\label{p:F3_tresciwy_to_poza_eta}
Let $T$ be defined as at the beginning of this subsection and let $(y_0, y_1, \ldots, y_n)$ be a $\CC$-basis of $T_1$.
 Assume that  $G=y_0^{d-3} P$ for some  natural number  $d \geq 5$ and $P \in T_3$. Define $f := P|_{y_0=1} = F_3 + F_2 + F_1 + F_0 \in R:=\CC[y_1, \ldots, y_n]$. If $f$ satisfies the following conditions:
 \begin{enumerate}[label=(\alph*)]
   \item $\Apolar(f^{\btd d})$ has Hilbert function $(1,6,6,1)$,
   \item $[\Spec \Apolar(f^{\btd d})] \notin \mathcal{H}ilb^{Gor, sm}_{14}(\AA^n)$,
\end{enumerate}
then $[G]\in \eta_{14} ( \nu_{d} (\PP^n)) \setminus \sigma_{14}( \nu_{d} (\PP^n))$.
\end{proposition}

\begin{proof}
  By Condition (a) we have $\dim_{\mathbb{C}}(R^*/\Ann(f^{\btd d})) = 14$. Therefore from
Theorem \ref{t:general_polynomial} (i)
\begin{equation*}
  \crr(G)=\crr{(\sum_{i=0}^3 y_0^{[d-i]} F^{\btd d}_i)} \leq 14.
\end{equation*} 
From the Border Apolarity Lemma \ref{p:border_apolarity}, if $[G] \in
\sigma_{14}( \nu_{d} (\PP^n))$ then there exists $J \subseteq \Ann(G)$ with
$[J] \in \Slip_{14, \mathbb{P}T_1} \subseteq
\operatorname{Hilb}_{T^*}^{h_{14}}$. Thus $[\Proj (T^* / J^{sat})] \in
\mathcal{H}ilb^{sm}_{14} (\PP^n)$.  The Hilbert function of $R^*/\Ann(F_3^{\btd d})$ is $(1,6,6,1)$ by \cite[Thm.~2.3 and the following remarks]{CJN15}. In particular, $F_3^{\btd d}$ is not a power of a linear form. From Theorem~\ref{t:general_polynomial}
(iii)  it follows that $J^{sat} = \Ann(f^{\btd d})^{hom}$, so
\[
  [\Spec(R^*/\Ann(f^{\btd d}))] \in \mathcal{H}ilb^{Gor,sm}_{14}(\AA^{n}).
\]
This contradicts Condition (b).
\end{proof}

Finally we present the proof of the  characterization of points of the second irreducible component of the cactus variety.

\begin{proof}[Proof of Theorem \ref{t:segre-veronese_map_general_n} and Corollary \ref{c:segre-veronese_map}]

We first prove Part (ii) of Theorem \ref{t:segre-veronese_map_general_n} for $\eta_{14}(\nu_d(\PP^n))$ defined as in Equation~\eqref{eq:eta_is_projection}. Let $\psi\colon \PP T_1\times \PP T_3\to \PP T_d$ be given by $([y_0], [P])\mapsto [y_0^{d-3}P]$ and let $q\colon (T_1\setminus \{0\})\times (T_3\setminus \{0\}) \to \PP T_1\times \PP T_3$ be the natural map. 
Let  
\begin{multline*}
C = \{(y_0, P) \in T_1\times T_3 \mid\text{ there exists a completion of } y_0 \text{ to a basis } (y_0,y_1,\ldots, y_n) \text{ of } T_1 \text{ such that } \\
\Apolar((P|_{y_0=1})^{\btd d}) \text{ has Hilbert function } (1,6,6,1) \}\text{.}
\end{multline*}
Note that the set from the statement is $\psi(q(C))$. We define
\begin{equation*}
D = \{(y_0, P) \in C\mid [\Spec \Apolar ((P|_{y_0=1})^{\btd d})]\notin \Hilb_{14}^{Gor, sm}(\AA^n)\}.
\end{equation*}
We claim that the set $C$ is irreducible, $D$ is dense in $C$, and that $\dim D = \dim C = 14n + 7$. In order to prove the claim, we consider the morphism $\varphi : GL(T_1) \times T_3 \to T_3$ given by a change of basis. Then we have a product morphism
\begin{equation*}
\tau : GL(T_1) \times T_3 \to T_1 \times T_3\text{, given by } (a, P) \mapsto (a(x_0),\varphi(a,P))\text{.}
\end{equation*}
Recall the sets $A$, and $B$ from Lemma \ref{l:D_i_is_dense_and_of_dimension_13n+5}. Let $\widehat{A}$ and $\widehat{B}$ be the affine cones over $A$ and $B$ with origins removed. Let $\chi:S_{\leq 3} \to T_3$ be the inverse of the $\CC$-linear isomorphism $T_3 \to S_{\leq 3}$ given by $P \mapsto (P|_{x_0=1})^{\btd d}$.
  We have $\tau(GL(T_1)\times \chi(\widehat{A})) = C$ and $\tau(GL(T_1)\times\chi(\widehat{B})) = D$. These follow from the definitions of the sets $A,B,C,D$ and the identity
  \[
  (\varphi(a, \chi(f))|_{a(x_0)=1})^{\btd d} = f(a(x_1), \ldots, a(x_n)) \text{ for every } f\in S_{\leq 3} \text{ and } a\in GL(T_1).
  \]
  
  It follows from Lemma \ref{l:D_i_is_dense_and_of_dimension_13n+5} that $C$ is irreducible, $D$ is dense in $C$, and $\dim D = \dim C = (n+1) + (13n+6) = 14n +7$. 
  By Proposition~\ref{p:F3_tresciwy_to_poza_eta} if $(y_0, P)\in D$ and $G=y_0^{d-3}P$ then $[G]\in \eta_{14} ( \nu_{d} (\PP^n))$. Hence we have the
  inclusion $\overline{\psi(q(C))} \subseteq \eta_{14}(\nu_d(\mathbb{P}^n))$.

Now we prove that in fact $\overline{\psi(q(C))} = \eta_{14}(\nu_{d} (\PP^n))$.
 It follows from Proposition \ref{p:dim_eta}  that for every $ d  \geq 5$ we have
 \[\dim(\eta_{14}(\nu_{d} (\PP^n))) \leq 14n+5 \leq \dim (q(C)) = \dim (\overline{q(C)}) = \dim \overline{\psi(q(C))}.\]
 The last equality follows from  \cite[Thm. 11.4.1]{RV2017}, since the fibers of $\psi$ are finite.
 Hence $\overline{\psi(q(C))} = \eta_{14}(\nu_{d} (\PP^n))$.
 This concludes the proof of Theorem \ref{t:segre-veronese_map_general_n}(ii).
 
 We proceed to the proof of Theorem \ref{t:segre-veronese_map_general_n}(i). By Theorem \ref{t:1661_summary}, the scheme $\mathcal{H}ilb^{Gor}_{14}(\PP^n)$ has two irreducible components. Thus, the variety $\kappa_{14}(\nu_d(\PP^n))$ has at most two irreducible components: $\sigma_{14}(\nu_d(\PP^n))$ and $\eta_{14}(\nu_d(\PP^n))$ by Equations \eqref{eq:kappa_is_projection}, \eqref{eq:sigma_is_projection}, \eqref{eq:eta_is_projection}. Let $[f]\in B$, where $B$ is as in Lemma \ref{l:D_i_is_dense_and_of_dimension_13n+5}. Then by Proposition~\ref{p:F3_tresciwy_to_poza_eta} we get $[f^{hom, d-3}]\in \kappa_{14}(\nu_d(\PP^n)) \setminus \sigma_{14}(\nu_d(\PP^n))$. It is enough to show that $\eta_{14}(\nu_d(\PP^n) \neq \kappa_{14}(\nu_d(\PP^n))$. By Part (ii) every $[G]\in \eta_{14}(\nu_d(\PP^n))$ is divisible by the $(d-3)$-rd power of a linear form. Therefore, 
$[x_0^d + x_1^d]\in \kappa_{14}(\nu_d(\PP^n))\setminus \eta_{14}(\nu_d(\PP^n))$.

   Now we prove Corollary \ref{c:segre-veronese_map}. The cactus variety $\kappa_{14}(\nu_d(\mathbb{P}^6))$ has two irreducible components by Part (i) of Theorem \ref{t:segre-veronese_map_general_n}. 
Assume that $n=6$. Then in the above notation the closure of $q(C)$ in $\PP T_1 \times \PP T_3$ has the maximal dimension $14\cdot 6+5=89$. Thus $\overline{q(C)} = \PP T_1\times \PP T_3$. It follows that $\eta_{14}(\nu_d(\PP^6)) = \overline{\psi(q(C))} = \psi(\PP T_1 \times \PP T_3)$.
\end{proof}

\begin{proposition}\label{p:l5q_of_cactus_rank_at_most_13}
Let $d\geq 4$ be an integer, $y_0\in T_1$, $Q\in T_2$. Define $G=y_0^{d-2} Q \in T_d$. If $[G]\in \kappa_{r}(\nu_d(\PP^n))$ for a positive integer $r$, then $[G]\in \sigma_{r}(\nu_d(\PP^n))$.
\end{proposition}
\begin{proof}
 Complete $y_0$ to a basis $(y_0,y_1,\dots,y_n)$ of $T_1$.
 Let $S=\CC[y_1,...,y_n]$ and $q=Q_2 + Q_1 + Q_0 \in S$ be such that $G= Q_2 y_0^{[d-2]} + Q_1 y_0^{[d-1]} + Q_0 y_0^{[d]}$. By Theorem \ref{t:general_polynomial}(ii) we have $\dim_{\CC} S^*/\Ann(q) = s$ for some $s\leq r$.  Therefore, 
 \[
 [\Proj T^*/\Ann(q)^{hom}] \in \Hilb_s(\PP^n)\text{.}
 \]
By \cite[Prop.~4.9]{CEVV09} this subscheme is smoothable.
Hence, it follows from Lemma \ref{l:smoothable_then_br_of_hom_at_most_r}  that $[G]\in \sigma_{r}(\nu_d(\PP^n))$.
\end{proof}

The following lemma gives a description of the set-theoretic difference of the cactus variety and the secant variety. We need it to give a clear proof of Theorem \ref{t:algorithm}.

\begin{lemma}\label{l:algorithm_form}
Let $d\geq 6, n\geq 6$. The point $[G] \in \kappa_{14}(\nu_{d}(\PP^n))$
does not belong to $\sigma_{14}(\nu_{d}(\PP^n))$ if and only if there exists a
linear form $y_0 \in T_1$, and $P \in T_3$ such that $G = y_0^{d-3}P$ and for any
completion of $y_0$ to a basis of $T_1$ we have:
  \begin{enumerate}[label=(\alph*)]
    \item $\Apolar((P|_{y_0 = 1})^{\btd d})$ has Hilbert function $(1,6,6,1)$,
    \item $[\Spec \Apolar((P|_{y_0 = 1})^{\btd d})] \notin \mathcal{H}ilb^{Gor, sm}_{14}(\AA^n)$.
   \end{enumerate}
\end{lemma}
\begin{proof}
  If $y_0 \in T_1$ and $P \in T_3$ are such that $G = y_0^{d-3} P$, and there exists
  a completion of $y_0$ to a basis $(y_0,\dots,y_n)$ of $T_1$, for which
  Conditions (a),(b) hold, we get 
  \begin{equation*}
    [G] \notin \sigma_{14}(\nu_{d}(\PP^n))
  \end{equation*}
  by Proposition \ref{p:F3_tresciwy_to_poza_eta}.

  Assume that $[G] \notin \sigma_{14}(\nu_{d}(\PP^n))$. Then by Theorem
  \ref{t:segre-veronese_map_general_n} there exists a linear form $y_0 \in T_1$ such that
  $y_0^{d-3} \mid G$. Using Proposition \ref{p:l5q_of_cactus_rank_at_most_13} we
  conclude that $G$ is not divisible by $y_0^{d-2}$. Hence we showed that $G = y_0^{d-3} P$
  for some $P \in T_3$. Extend $y_0$ to a basis
  $(y_0,y_1,\dots,y_n)$. Let $f = P|_{y_0 = 1}$. Suppose $f = F_3 + F_2 + F_1 + F_0$.

  Now we prove Conditions (a), (b) hold. It follows from the definition of $()^{\btd d}$ that $f^{\btd d} = F^{\btd d}_3+F_2^{\btd d}+F_1^{\btd d}+F_0^{\btd d} \in \CC[y_1,\ldots,
  y_n]$. 
We have 
\[
  G=\sum_{i=0}^3 y_0^{[d-i]} F^{\btd d}_i.
\]
  By Lemma \ref{l:ideals_agree_in_low_degree} (i)
  \begin{equation*}
    \Ann(f^{\btd d})^{hom} \subseteq \Ann(G)\text{.}
  \end{equation*}
  If  $\dim_\CC (\Apolar(f^{\btd d})) \leq 13$, then $\crr(G) \leq 13$ by the Cactus Apolarity Lemma \ref{p:cactus_apolarity}, since $\Ann(f^{\btd d})^{hom}$ is saturated by Lemma \ref{l:homogenization_is_saturated}. Therefore, $[G]\in \kappa_{13}(\nu_d(\PP^n)) = \sigma_{13}(\nu_d(\PP^n)) \subseteq \sigma_{14}(\nu_d(\PP^n))$, a contradiction.
  
  From Theorem \ref{t:general_polynomial}(ii) we obtain $\dim_\CC (\Apolar(f^{\btd d})) \leq 14$. We proved that $\dim_\CC (\Apolar(f^{\btd d})) = 14$.
  Since we assumed that $[G] \not \in \sigma_{14} (\nu_d (\PP^n))$, it follows from Lemma \ref{l:smoothable_then_br_of_hom_at_most_r} that $\Spec \Apolar(f^{\btd d})$ is not smoothable.
  This implies Condition (b) holds.
  By \cite[Thm. 2.3]{CJN15} and \cite[Prop. 6.11]{CJN15} the algebra $\Apolar(f^{\btd d})$ has Hilbert function $(1,6,6,1)$.
  Thus we proved Condition (a) holds.
\end{proof}
Steps 2--5 of the algorithm from Theorem \ref{t:algorithm} check whether $G$ is of the form as in Lemma \ref{l:algorithm_form}.

\begin{proof}[Proof of Theorem \ref{t:algorithm}]
Assume that $[G ]\notin \sigma_{14}(\nu_{d}(\PP^n))$. Then there exist a basis
$(y_0,\ldots, y_n)$ of $T_1$ and $P\in T_3$ as in Lemma
\ref{l:algorithm_form}. Let $f=P|_{y_0=1}$. It follows from the definition of $()^{\btd d}$ that $f^{\btd d} = F^{\btd d}_3+F_2^{\btd d}+F_1^{\btd d}+F_0^{\btd d} \in
\CC[y_1,\ldots, y_n]$. Therefore, $G  =
y_0^{[d-3]}F^{\btd d}_3+y_0^{[d-2]}F^{\btd d}_2+y_0^{[d-1]}F_1^{\btd d}+y_0^{[d]}F_0^{\btd d}$. By Lemma
\ref{l:ideals_agree_in_low_degree}(ii), we have $\Ann (G)_{\leq d-3} =
(\Ann(f^{\btd d})^{hom})_{\leq d-3}$. Moreover, by Lemma
\ref{l:extension_of_annihilator_generated_in_low_degrees},
\begin{equation}\label{eq:anihilators}
  ((\Ann(f^{\btd d})^{hom})_{\leq d-3}) = \Ann(f^{\btd d})^{hom}
\end{equation}
(the assumptions of the lemma are satisfied since $\Apolar(F_3^{\btd d})$ and $\Apolar(f^{\btd d})$ have the same Hilbert function by \cite[Thm.~2.3 and the following remarks]{CJN15}).
Therefore, we have
\[
\mathfrak{a} = \sqrt{(\Ann(G )_{\leq d-3})} = \sqrt{\Ann(f^{\btd d})^{hom}} = (\beta_1,\ldots, \beta_n),
\]
where $\beta_1,\ldots, \beta_n \in T^*_1$ are dual to $y_1,\ldots, y_n\in T_1$. This shows that if the $\CC$-linear space $\big(\sqrt{(\Ann(G)_{\leq d-3})}\big)_1$ is not $n$-dimensional, then $[G] \in \sigma_{14}(\nu_{d}(\PP^n))$. Therefore, in that case, algorithm stops correctly at Step 2.

Assume that the algorithm did not stop at Step 2. Then if $G$ is of the form as in Lemma \ref{l:algorithm_form}, then $y_0$ divides $G$ exactly $d-3$ times. Otherwise $[G]\in \sigma_{14}(\nu_{d}(\PP^n))$ and the algorithm stops correctly at Step 3.

Assume that the algorithm did not stop at Step 3. Then the algorithm does not stop at Step 4 if and only if Condition (a) of Lemma \ref{l:algorithm_form} is fulfilled. Therefore, if the Hilbert function of $R^*/\Ann(f^{\btd d})$ is not equal to $(1,6,6,1)$, the algorithm stops correctly at Step 4. 

Assume that the algorithm did not stop at Step 4. Then $P$ satisfies Condition
(a) from Lemma \ref{l:algorithm_form}. Hence $[G]$ is in
$\sigma_{14}(\nu_{d}(\PP^n))$ if and only if $P$ does not satisfy Condition
(b). Using Lemma \ref{l:dim_H1661}, this is equivalent to
\[
  \dim_{\mathbb{C}}\operatorname{Hom}_{R^*}(I, R^*/I) > 14n-8\text{.}
\]
The left term is the dimension of the tangent space to the Hilbert scheme $\mathcal{H}ilb^{Gor}_{14}(\AA^n)$ at the point $[\Spec R^*/I]$ (see \cite[Prop. 2.3]{H09} or \cite[Thm. 18.29]{MS04}).
\end{proof}

\begin{remark}
The algorithm is stated for $d \geq 6$ even though it is based on Theorem
\ref{t:segre-veronese_map_general_n} which works for $d \geq 5$. The reason for
this is that we needed $d\geq 6$ to obtain Equation \eqref{eq:anihilators} and for Lemma \ref{l:algorithm_form} to work. We
do not know a counterexample for the algorithm in case $d=5$.
\end{remark}

Equations defining the cactus variety $\kappa_{14}(\nu_6(\mathbb{P}^n))$ for $n \geq 6$ are unknown and there is no example of an explicit equation of the secant variety $\sigma_{14}(\nu_6(\mathbb{P}^n))$ which does not vanish on the cactus variety. We present some known results about 14-th secant and cactus varieties of  Veronese embeddings of $\PP^6$.
    
\begin{remark} Let $V$ be a $7$-dimensional complex vector space. The
  catalecticant minors define a subscheme of $\mathbb{P}(\Sym^6V)$, one of whose
  irreducible components is the secant variety $\sigma_{14}(\nu_6(\mathbb{P}V))$
  (see \cite[Thm. 4.10A]{IK06}). Moreover, these equations are known to vanish
  on the cactus variety $\kappa_{14}(\nu_6(\mathbb{P}V))$ (see
  \cite[Prop. 3.6]{BB14},  or \cite{Gal16}). Example
  \ref{e:not_enough_equations} shows that the catalecticant minors do not
  define $\kappa_{14}(\nu_6(\mathbb{P}V))$ set-theoretically. However, if we
  consider the $d$-th Veronese for $d \geq 28$, then the catalecticant minors
  are enough to define $\kappa_{14}(\nu_d(\mathbb{P}V))$ set-theoretically, see
  \cite[Thm. 1.5]{BB14}. The article \cite{LO13} gives an extensive list of
  results on equations of secant varieties but in the case of $\sigma_{14}(\nu_6(\mathbb{P}V))$
  it does not improve the result in $\cite{IK06}$.
\end{remark}

\begin{example}\label{e:not_enough_equations}
  Let $F=x_0^{6}+x_1^{2}x_2^{2}x_3^{2}+x_4^{3}x_5^{2}x_6\in T =
  \mathbb{C}[x_0,\ldots,x_6]$. Then Hilbert function of $T^*/\Ann(F)$ is
  $(1,7,12,14,12,7,1)$ but there is only one minimal homogeneous generator of
  $\Ann(F)$ in degree 4. Therefore, there is no homogeneous ideal $J$ in $T^*$ such that  $T^*/J$ has a $(14,7)$-standard Hilbert function and $J$ is contained in $\Ann(F)$. 
  Thus $\operatorname{bcr}(F) > 14$ by the Weak Border Cactus Apolarity Lemma \ref{p:weak_border_cactus_apolarity_lemma},
  even though the Hilbert function of $T^*/\Ann(F)$ is bounded by $14$. 
\end{example}

\section[\texorpdfstring{(8,3)-th Grassmann cactus variety of Veronese embeddings of $\mathbb{P}^n$}{(8,3)-th Grassmann cactus variety of Veronese embeddings of projective space of dimension 4}]{(8,3)-th Grassmann cactus variety of  Veronese embeddings of $\mathbb{P}^n$}\label{s:83grassmann}

In this section we show that the Grassmann cactus variety
$\kappa_{8,3}(\nu_d(\mathbb{P}^n))$ has two irreducible components for $d \geq 5$ and $n\geq 4$, one of which is the Grassmann secant variety $\sigma_{8,3}(\nu_d(\PP^n))$ and the other one is described in Theorem \ref{t:segre-veronese_map_143_general_case}. 
Furthermore, we present an algorithm (Theorem \ref{t:algorithm_143}) for deciding whether $[V] \in
\kappa_{8,3}(\nu_{d}(\mathbb{P}^n))$  is in $
\sigma_{8,3}(\nu_{d}(\mathbb{P}^n))$.

We will assume that $\Bbbk = \CC$ because of technical reasons.  
In that case, the graded dual ring of a polynomial ring is isomorphic to a polynomial ring.

It follows from \cite{CEVV09} that for $n \geq 4, k \geq 1$ and $d \geq 2$ the Grassmann cactus variety $\kappa_{8,k}(\nu_d(\mathbb{P}\CC^{n+1}))$ has at most two irreducible
components. 
The main result of this section is Theorem~\ref{t:segre-veronese_map_143_general_case}, analogous to Theorem~\ref{t:segre-veronese_map_general_n}.
In Part (i)  we verify that for all $d\geq 5$ and all $n \geq 4$ the Grassmann cactus variety $\kappa_{8,3}(\nu_d(\PP \CC^{n+1}))$ is reducible. The main result
is Part (ii) which describes the irreducible component other than $\sigma_{8,3}(\nu_d(\PP \CC^{n+1}))$.

\begin{theorem}\label{t:segre-veronese_map_143_general_case}
  Let $n\geq 4$ and $d\geq 5$ be integers and consider the polynomial ring $T=\CC[x_0,\ldots, x_n]$. 
 \begin{enumerate}[label=(\roman*)]
\item The Grassmann cactus variety $\kappa_{8,3}(\nu_d(\PP T_1))$ has two irreducible components, one of which is the Grassmann secant variety $\sigma_{8,3}(\nu_d(\PP T_1))$, and we denote the other one by $\eta_{8,3}(\nu_d(\mathbb{P}T_1))$.
\item The irreducible component $\eta_{8,3}(\nu_d(\PP T_1))$ is the closure of the following set
  \begin{multline*}%
    \{[y_0^{d-2} U] \in \Gr(3, T_d) \mid y_0 \in T_1\setminus \{0\}, U \in \Gr(3, T_2) \text{, and there exists
    a completion of } y_0 \text{ to a basis } \\ (y_0,y_1,\ldots, y_n) \text{ of } T_1 \text{ such that } \Apolar((U|_{y_0=1})^{\btd d}) \text{ has
Hilbert function } (1,4,3)\}.
  \end{multline*}
  \end{enumerate}
\end{theorem}

For $n = 4$, Theorem \ref{t:segre-veronese_map_143_general_case} has the following simple form.

\begin{corollary}\label{c:segre-veronese_map_143}
  Let $d\geq 5$, the Grassmann cactus variety $\kappa_{8,3}(\nu_{d}(\mathbb{P}^4))$ has two irreducible components: the Grassmann secant variety $\sigma_{8,3}(\nu_{d}(\PP^4))$ and the variety $\eta_{8,3}(\nu_{d}(\PP^4))$ consisting of $3$-dimensional vector spaces divisible by a $(d-2)$-nd power of a linear form.
\end{corollary}



Let $\mathcal{H}ilb_{r}(X)$, where $X=\mathbb{A}^n$ or $\mathbb{P}^n$, denote
the Hilbert scheme of $r$ points on $X$. Since we shall use
the results of Cartwright, Erman, Velasco and Viray from \cite{CEVV09} in this section,
we give a brief summary.
\begin{theorem}[{\cite[Thm. 1.1]{CEVV09}}]\label{t:143_summary}
  We have the following:
  \begin{enumerate}[label=(\roman*)]
    \item the scheme $\mathcal{H}ilb_{r}(\AA^n)$ is irreducible for $r < 8$ and any $n\in \mathbb{N}$,
    \item the scheme $\mathcal{H}ilb_{8}(\AA^n)$ is reducible if and only if $n\geq 4$,
    \item if the scheme $\mathcal{H}ilb_{8}(\AA^n)$ is reducible, it has
      two irreducible components $\mathcal{H}ilb_{8}^{sm}(\AA^n)$, the
      closure of the set of smooth schemes, and $\mathcal{H}^n_{143,af}$ the
      closure of the set of local algebras with local Hilbert function
      $(1,4,3)$.
  \end{enumerate}
\end{theorem}
\begin{remark}\label{r:why_83}
  By Theorem \ref{t:143_summary}, we know that $\sigma_{r,k}(\nu_d(\PP^n)) = \kappa_{r,k}(\nu_d(\PP^n))$ for $r \leq 7$, and any $k,n,d$, and that
  $\sigma_{8,k}(\nu_d (\PP^n )) = \kappa_{8,k}(\nu_d(\PP^n))$ for $n \leq 3$, and any $k,d$. In addition, we claim that $\sigma_{8,2}(\nu_d(\PP^n)) = \kappa_{8,2}(\nu_d(\PP^n))$ for any $n$. We sketch the proof. All local algebras of length at most $8$ and socle dimension at most $2$ are smoothable by Theorem \ref{t:143_summary}. Hence the claim follows from the fact that $\kappa_{r,k}(\nu_d(\PP^n))$ is the closure of the following set
  \begin{equation*}
    \{R \hookrightarrow \mathbb{P}^n \mid \operatorname{length} R \leq r\text{, }
    H^0(R,\mathcal{O}_R) \text{ \emph{is a product of local algebras of socle dimension at most} } k\}
  \end{equation*}
  (a generalization of \cite[Prop.~2.2]{BB14}). Detailed proof of this fact is
  outside the main interests of this article, hence we skip it.

  It follows from the above discussion that the number $k=3$ is the smallest one such that the variety $\kappa_{8,k}(\nu_d (\PP^n))$ can be reducible for some $d,n$. 
  That is why we focus on studying $\kappa_{8,3}(\nu_d(\PP^n))$ for $n \geq 4$.
\end{remark}

Part (iii) of Theorem \ref{t:143_summary} is why we analyze algebras with local
Hilbert function $(1,4,3)$ in the next section. It follows from the theory of
Macaulay's inverse systems (\cite[Thm. 21.6]{Eis95}) that every such algebra
is the apolar algebra of a linear subspace of three (non-necessarily
homogeneous) quadrics.

\subsection[\texorpdfstring{The set of subspaces with Hilbert function $(1,4,3)$}{The set of subspaces with Hilbert function (1,4,3)}]{The set of subspaces with Hilbert function $(1,4,3)$}

In Lemma
\ref{l:equivalent_description_of_143}, we give a useful characterization of
subspaces $W$ of a polynomial ring such that the Hilbert function of $\Apolar(W)$ is $(1,4,3)$. Then we establish Lemma
\ref{l:D_i_is_dense_and_of_dimension_7n+8} about topological properties of the set of such subspaces.

In this subsection $S^* = \CC[\alpha_1,\dots,\alpha_n]$, and $S =
\CC[x_1,\dots,x_n]$ is its graded dual. We assume that $n\geq 4$. Given an integer $i$, and a linear subspace $W \subseteq S$, we denote by $W_i$ the image of the projection of $W$ onto the $i$-th graded part.

\begin{lemma}\label{l:equivalent_description_of_143}
For $[W]\in \Gr(3, S_{\leq 2})$ the following are equivalent:
\begin{enumerate}[label = {(\alph*)}]
\item $\Apolar(W)$ has Hilbert function $(1,4,3)$,
\item $\dim_\CC W_2 = 3$,  $[W] \in \Gr(3, \Sym^2 U \oplus S_{\leq 1})$ for some $[U]\in \Gr(4, S_1)$ and $H(\Apolar(W_2),1) = 4$,
\item $\Apolar (W_2)$ has Hilbert function $(1,4,3)$.
\end{enumerate}
\end{lemma}
\begin{proof}
Conditions $(b)$ and $(c)$ are equivalent. We shall show that Conditions $(a)$ and $(c)$ are equivalent.
Observe that $H(\Apolar(W),2) = 3$ if and only if $\dim_\CC W_2 =  3$ since $H(\Apolar(W),2) = H(\Apolar(W_2), 2)$.

Therefore, we are left to show that $H(\Apolar(W),1) = 4$ if and only if $H(\Apolar(W_2),1) = 4$.
By Lemma~\ref{l:calculation}, we obtain $H(\Apolar(W), 1) = \codim_{S_1^*} (E_1)$, where 

\[
E_1=\{ \theta_1\in S^*_1 \mid \text{ there exists } \theta_{\geq 2} \in S^*_{\geq 2} \text{ such that } \theta_1+\theta_{\geq 2} \in \Ann(W) \}.
\]
We will show that $E_1 = \Ann(W_2)_1$.

\noindent Let $W=\langle Q_j+L_j+C_j \mid j\in \{1,2,3\}, \text{ } Q_j \in S_2,  L_j \in S_1 \text{ and }  C_j \in S_0 \rangle$.
Assume that $\theta_1 \in E_1$ and let $\theta_1+\theta_{\geq 2} \in \Ann(W)$ for some $\theta_{\geq 2} \in S^*_{\geq 2}$. Then for $j\in \{1,2,3\}$
\[
0=(\theta_1 + \theta_{\geq 2}) \lrcorner (Q_j+L_j+C_j) = (\theta_1 \lrcorner Q_j) + (\theta_1 \lrcorner L_j + \theta_{\geq 2} \lrcorner Q_j),
\]
so $\theta_1 \lrcorner Q_j=0$ for $j\in\{1,2,3\}$.

Now assume that $\theta_1\in \Ann(W_2)$. 
Since $\dim_\CC W_2 = 3$, there is $\theta_2 \in S^*_{2}$ such that $\theta_2 \lrcorner Q_j = - \theta_1\lrcorner L_j$ for $j\in \{1,2,3\}$. Then $\theta_1+\theta_2\in \Ann(W)$, so $\theta_1\in E_1$.

\end{proof}

\begin{lemma}\label{l:D_i_is_dense_and_of_dimension_7n+8}
The following subset 
$$
A=\{[W]\in \Gr(3,S_{\leq 2}) \mid \Apolar(W) \text{ has Hilbert function }(1,4,3)\}.
$$
is irreducible, of dimension $7n+8$, and locally closed.
Moreover the set
$$
B=\{[W]\in A \mid  [\Spec \Apolar(W)] \notin \Hilb_8^{sm}(\mathbb{A}^n)\}
$$
is dense in $A$.
\end{lemma}
\begin{proof}
Consider 
\[
\mathcal{A}= \{([U],[W]) \in \Gr(4, S_1) \times \Gr(3, S_{\leq 2}) \mid  [W] \in \Gr(3, \Sym^2 U \oplus S_{\leq 1})\}\text{.}
\]
and 
\[
\mathcal{A}^0= \{([U],[W]) \in \mathcal{A} \mid H(\Apolar(W_2),1)=4\}\text{.}
\]
We have a pullback diagram
\begin{center}
\begin{tikzcd}
\mathcal{A} \arrow[r] \arrow[d] & \operatorname{Fl}(3,n+11 , S_{\leq 2}) \arrow[d] \\
\operatorname{Gr}(4, S_1) \arrow[r] & \Gr(n+11, S_{\leq 2})
\end{tikzcd}
\end{center}
where $\operatorname{Fl}(3,n+11 , S_{\leq 2})$ is the flag variety parametrizing flags of subspaces  $M\subseteq N \subseteq S_{\leq 2}$ with $\dim_\CC M = 3$, $\dim_\CC N = n + 11$ and the lower horizontal map sends $[U]$ to $[\Sym^2 U \oplus S_{\leq 1}]$.

The varieties $\mathcal{A}$ and $\Gr(4, S_1)$ are projective. Moreover, the left vertical map is surjective and its fibers are irreducible and
isomorphic to $\Gr(3, n+11)$. Since $\Gr(3, n+11)$ is irreducible, it follows from \cite[Thm.~1.25 and 1.26]{Sha13} that $\mathcal{A}$ is irreducible and of dimension $4(n-4)+3(n+8) = 7n+8$.

We will show that $\mathcal{A}^0$ is open in $\mathcal{A}$. Consider the subset
\[
\mathcal{B}= \{([U],[W]) \in \Gr(4, S_1) \times \Gr(3, S_{\leq 2}) \mid  H(\Apolar(W_2),1) \geq 4 \}\text{.}
\]
Observe that $\mathcal{A}^0= \mathcal{A} \cap \mathcal{B}$, therefore it is enough to show that $\mathcal{B}$ is open in $\Gr(4, S_1) \times \Gr(3, S_{\leq 2})$. 
Let
\[
 \mathcal{C} = \{ [W] \in \Gr(3, S_{\leq 2}) \mid H(\Apolar(W_2),1) \geq 4 \}.
\]
It is enough to show that $\mathcal{C}$ is open in $\Gr(3, S_{\leq 2})$.
Let 
\[
\mathcal{D}= \{([U],[W]) \in \Gr(3, S_1) \times \Gr(3, S_{\leq 2}) \mid  [W] \in \Gr(3, \Sym^2(U)\oplus S_{\leq 1})\}\text{,}
\]
and $\rho_2:\Gr(3, S_1) \times \Gr(3, S_{\leq 2}) \to \Gr(3, S_{\leq 2})$ be the natural projection.
Observe that the complement of $\mathcal{C}$ in $\Gr(3, S_{\leq 2})$ is equal to $\rho_2 (\mathcal{D})$ which is closed since $\mathcal{D}$ is projective. This concludes the proof that $\mathcal{A}^0$ is open in $\mathcal{A}$.

By Lemma \ref{l:equivalent_description_of_143} we have $A = \pi_2(\mathcal{A}^0) \cap \mathcal{F}$, where 
\[
\mathcal{F}= \{[W] \in \Gr(3, S_{\leq 2}) \mid \dim_\CC W_2 = 3\}
\]
and $\pi_2 \colon \Gr(4, S_{1}) \times \Gr(3, S_{\leq 2})  \to \Gr(3, S_{\leq 2})$ is the projection.

Since $\pi_2|_{\mathcal{A}^0}\colon \mathcal{A}^0 \to \pi_2(\mathcal{A}^0)$ has a finite fiber over a general point, it follows from \cite[Thm.~11.4.1]{RV2017} that $\pi_2(\mathcal{A}^0)$ is irreducible and of dimension $7n+8$. The subset $\mathcal{F}\subseteq \Gr(3, S_{\leq 2})$ is open and $\pi_2(\mathcal{A}^0)\cap \mathcal{F}$ is non-empty, so $A=\pi_2(\mathcal{A}^0)\cap \mathcal{F}$ is irreducible and of dimension $7n+8$.

We know that $A=\pi_2 (\mathcal{A}) \cap \mathcal{C} \cap \mathcal{F}$, so $A$ is locally closed since $\pi_2 (\mathcal{A})$ is closed and $\mathcal{C} $, $\mathcal{F}$ are open.
Therefore, we have a morphism $\mu : A \to \Hilb_{8}(\AA^n)$ given on closed points by $[W] \mapsto [\Spec S^*/\Ann(W)]$, see Theorem~\ref{t:morphism}. 

By Theorem~\ref{t:143_summary}, the scheme $\Hilb_{8}(\mathbb{\AA}^n)$
has two irreducible components
$\Hilb_{8}^{sm}(\mathbb{\AA}^n)$ and $\mathcal{H}^n_{143,af}$.
We obtain $B = \mu^{-1}(\mathcal{H}_{143,af}^n \setminus \Hilb_{8}^{sm} (\AA^n))$,
so it is open in $A$. We claim that $B$ is non-empty. Indeed, consider the
subspace $W = \langle x_2 x_4, x_1 x_3 ,x_2 x_3 - x_1 x_4\rangle \subseteq
S_{\leq 2}$. By Lemma \ref{l:equivalent_description_of_143} we have that $[W] \in A$. Furthermore, we can calculate that
\begin{equation*}
  \Ann(W) = (\alpha_1^2, \alpha_2^2, \alpha_3^2, \alpha_4^2, \alpha_1 \alpha_2, \alpha_3 \alpha_4, \alpha_1\alpha_4 + \alpha_2 \alpha_3,\alpha_5, \alpha_6,\dots,\alpha_n),
\end{equation*}
and therefore $\Apolar W$ is non-smoothable, see \cite[the proof of Prop. 5.1]{CEVV09}. This finishes the proof of the claim. Since $B$ is open and non-empty and $A$ is
irreducible, it follows that $B$ is dense in $A$.
\end{proof}
\subsection{Proofs of the main theorems}
We will consider the polynomial ring $T^* = \CC
[\alpha_0,\alpha_1,\ldots,\alpha_n]$, and its graded dual $T =
\CC[x_0,x_1,\dots,x_n]$,  where $n \geq 4$. Since we assume $\kk = \CC$, the
graded dual ring $T$ is isomorphic to a polynomial ring. Recall Definitions
\ref{d:prime_operator}, \ref{d:homogenization_of_space}.

Our goal is to characterize for $d \geq 5$ and $ n \geq 4$ the closure of the set-theoretic difference between the cactus variety $\kappa_{8,3}(\nu_d(\PP T_1))$ and the secant variety $\sigma_{8,3}(\nu_d(\PP T_1))$. 
For $n=4$ and $d\geq 5$ this closure consists of points $[V]\in \Gr(3,T_d)$ with $V$ divisible by $(d-2)$-nd power of a linear form. However for $n > 4$ the situation is more complicated.

We start with showing that points of $\Gr(3, T_d)$  corresponding to subspaces divisible by $(d-1)$-st power of a linear form are in the Grassmann secant variety $\sigma_{8,3}(\nu_d(\PP T_1))$.
\begin{proposition}\label{p:l5l_of_cactus_rank_at_most_4_143}
Let $d\geq 2$ and $n \geq 4$ be integers, $y_0\in T_1$ and $[U]\in \Gr(3, T_1)$. Define $V=y_0^{d-1} U \in \Gr(3, T_d)$. Then $\crr(V) \leq 4$, so $[V] \in \kappa_{4,3}(\nu_d  (\PP T_1)) = \sigma_{4,3}(\nu_d  (\PP T_1)) \subseteq \sigma_{8,3}(\nu_d  (\PP T_1))$. 
\end{proposition}
\begin{proof}
Up to a linear change of variables, $V$ is of one of the following forms
\begin{enumerate}
\item $V= \langle x_0^{d-1}x_1, x_0^{d-1}x_2, x_0^{d-1}x_3 \rangle$ or
\item $V= \langle x_0^d, x_0^{d-1}x_1, x_0^{d-1}x_2 \rangle$.
\end{enumerate}
Then $V = W^{hom, d_2}$ for $d_2=d-1$, where $W$ is correspondingly
\begin{enumerate}
\item $W = \langle x_1, x_2, x_3 \rangle$ or
\item $W = \langle 1, x_1, x_2 \rangle$.
\end{enumerate}
In either case, $\dim_\CC S^* /\Ann(W) \leq 4$, so $\crr(V) = \crr(W^{hom, d_2}) \leq 4$ by Theorem \ref{t:general_subspace}(i).
\end{proof}

For $d \geq 2$ we will define a subset $\eta_{8,3}(\nu_d(\PP^n))$ of the Grassmann cactus variety $\kappa_{8,3}(\nu_d(\PP^n))$. Later, in Theorem \ref{t:segre-veronese_map_143_general_case}, it will be shown that for $d\geq 5$
\[
\kappa_{8,3}(\nu_d(\PP^n)) = \sigma_{8,3}(\nu_d(\PP^n)) \cup \eta_{8,3}(\nu_d(\PP^n))
\]
is the decomposition into irreducible components.

Consider the following rational map $\varphi$, which assigns to a scheme
  $R$ its projective linear span $\langle v_d(R) \rangle$ 
  \[\begin{tikzcd}
      \varphi: \Hilb_{8}(\mathbb{P}^n) \ar[dashed]{r} & \Gr(8, \Sym^{d} \mathbb{C}^{n+1})\text{.}
  \end{tikzcd}\]
  Let $U \subseteq \Hilb_{8}(\mathbb{P}^n)$ be a dense open subset on which $\varphi$ is regular.
  Consider the projectivized incidence bundle $\mathbb{P}\mathcal{S}$ over the Grassmannian $\Gr(8,\Sym^{d} \mathbb{C}^{n+1})$, given as a set by
  \begin{equation*}
    \mathbb{P}\mathcal{S} = \{ ([V_1],[V_2]) \in \Gr(8,\Sym^{d} \mathbb{C}^{n+1}) \times
  \Gr(3, \Sym^{d} \mathbb{C}^{n+1}) \mid V_2 \subseteq V_1 \}\text{,} 
  \end{equation*}
  together with the inclusion
  $i: \mathbb{P}\mathcal{S} \hookrightarrow
  \Gr(8,\Sym^{d} \mathbb{C}^{n+1}) \times
  \Gr(3, \Sym^{d} \mathbb{C}^{n+1})$.
  We pull the commutative diagram
  \[\begin{tikzcd}
      \mathbb{P}\mathcal{S} \ar[hookrightarrow]{rr}{i}\ar{dr}{\pi} & &  \Gr(8,\Sym^{d} \mathbb{C}^{n+1}) \times
  \Gr(3, \Sym^{d} \mathbb{C}^{n+1}) \ar{dl}{\pr_1} \\
      & \Gr(8,\Sym^{d} \mathbb{C}^{n+1}) 
  \end{tikzcd}\]
  back along $\varphi$ to $U$, getting the commutative diagram
  \[\begin{tikzcd}
      \varphi^* (\mathbb{P}\mathcal{S}) \ar[hookrightarrow]{rr}{\varphi^* i}\ar{dr}{\varphi^*\pi} & &  U \times \Gr(3,\Sym^{d} \mathbb{C}^{n+1})  \ar{dl}{\pr_1} \\
      & U\text{.}
  \end{tikzcd}\]
  Let $Y$ be the closure of $\varphi^*(\mathbb{P}\mathcal{S})$ inside
  $\Hilb_{8}(\mathbb{P}^n)\times \Gr(3,\Sym^{d} \mathbb{C}^{n+1})$. The
  scheme $Y$ has two irreducible components, $Y_1$ and $Y_2$, corresponding to
  two irreducible components of $\Hilb_{8}(\mathbb{P}^n)$, the schemes
  $\Hilb_{8}^{sm}(\mathbb{P}^n)$ and $\mathcal{H}_{143}$, respectively, see Theorem \ref{t:143_summary}. Then for $d \geq 2$
  \begin{align}
    \kappa_{8,3}(\nu_d(\mathbb{P}^n)) & = \pr_2(Y) \label{eq:kappa_is_projection_143},\\
    \sigma_{8,3}(\nu_{d}(\mathbb{P}^n)) &= \pr_2(Y_1)\text{, and we define} \label{eq:sigma_is_projection_143}\\
    \eta_{8,3}(\nu_{d}(\mathbb{P}^n)) &:= \pr_2(Y_2)\label{eq:eta_is_projection_143}\text{.}
  \end{align}

In the following proposition we bound from above the dimension of the irreducible subset $\eta_{8,3}(\nu_d(\PP^n))$ by $8n+8$. 
Later, in Theorem \ref{t:segre-veronese_map_143_general_case}, we will identify a $(8n+8)$-dimensional subset of $\kappa_{8,3}(\nu_d(\PP^n))\setminus \sigma_{8,3}(\nu_d(\PP^n))$. We will be able to conclude that the closure of this subset is $\eta_{8,3}(\nu_d(\PP^n))$.

\begin{proposition}\label{p:dim_eta_143}
 The dimension of $\eta_{8,3}(\nu_{d} (\PP^n))$ is at most $8n+8$.
\end{proposition}
\begin{proof}
  We have the following commutative diagram

\begin{center}
\begin{tikzcd}[column sep = small, row sep = small]
\Gr(3,\Sym^{d} \mathbb{C}^{n+1}) \supseteq \sigma \cup \eta & Y_1\cup Y_2 \arrow[l] \arrow[d, "\chi"] \arrow[r, dashed] & \PP \mathcal{S} \arrow[d] \\
\mathcal{H}ilb_{8} (\PP^n) \arrow[r, equal] & \mathcal{H}ilb_{8}^{sm} (\PP^n) \cup \mathcal{H}_{143} \arrow[r, dashed] &  \Gr(8,\Sym^{d} \mathbb{C}^{n+1}){,}
\end{tikzcd}
\end{center}
where $\sigma$ and $\eta$ denote $\sigma_{8,3}( \nu_{d} (\PP^n))$ and $\eta_{8,3} ( \nu_{d} (\PP^n))$, respectively, and $\chi:Y_1\cup Y_2 \to
\Hilb_{8}(\mathbb{P}^n)$ is the projection. 
Then $\dim \eta_{8,3}(\nu_{d} (\PP^n)) \leq \dim (Y_2)
  = \dim   \mathcal{H}_{143} + 15$, where 15 is the dimension of the general fiber of the map $\chi|_{Y_2}: Y_2
  \to \mathcal{H}_{143}$. 
  It follows from \cite[Thm. 1.1]{CEVV09} that $\dim   \mathcal{H}_{143}=8n-7$ and therefore, $\dim \eta_{8,3}(\nu_{d} (\PP^n)) \leq 8n+8$.
\end{proof}

In the rest of the section we use the notation $W^{\btd d}$ from Definition \ref{d:prime_operator}. 

In the following proposition
we identify many points from the Grassmann cactus variety which are outside of the Grassmann secant variety. 
In fact, the closure of the set of these points  is the second irreducible component of the Grassmann cactus variety. This will be established in Theorem \ref{t:segre-veronese_map_143_general_case}.

\begin{proposition}\label{p:F3_tresciwy_to_poza_eta_143}
Let $T$ be defined as at the beginning of this subsection and let $(y_0, y_1, \ldots, y_n)$ be a $\CC$-basis of $T_1$.
 Assume that  $V=y_0^{d-2} U$ for some  natural number  $d \geq 5$ and $[U] \in \Gr(3, T_2)$. Define $[W] := [U|_{y_0=1}]\in \Gr(3, R_{\leq 2})$, where $R:=\CC[y_1, \ldots, y_n]$. If $W$ satisfies the following conditions:
 \begin{enumerate}[label=(\alph*)]
   \item $\Apolar(W^{\btd d})$ has Hilbert function $(1,4,3)$,
   \item $[\Spec \Apolar(W^{\btd d})] \notin \mathcal{H}ilb^{sm}_{8}(\AA^n)$,
\end{enumerate}
then $[V]\in \eta_{8,3} ( \nu_{d} (\PP^n)) \setminus \sigma_{8,3}( \nu_{d} (\PP^n))$.
\end{proposition}

\begin{proof}

  By Condition (a) we have $\dim_{\mathbb{C}}(R^*/\Ann(W^{\btd d})) = 8$. Therefore, from
Theorem \ref{t:general_polynomial} (i)
\begin{equation*}
  \crr(V)=\crr ({(W^{\btd d})^{hom, d-2}}) \leq 8.
\end{equation*} 
From the Border Apolarity Lemma \ref{p:border_apolarity}, if $[V] \in
\sigma_{8,3}( \nu_{d} (\PP^n))$ then there exists $J \subseteq \Ann(V)$ with
$[J] \in \Slip_{8, \mathbb{P}T_1} \subseteq
\operatorname{Hilb}_{T^*}^{h_{8}}$. Thus $[\Proj (T^* / J^{sat})] \in
\mathcal{H}ilb^{sm}_{8} (\PP^n)$. From Theorem \ref{t:general_subspace}
(iii)  it follows that $J^{sat} = \Ann(W^{\btd d})^{hom}$, so
\[
  [\Spec(R^*/\Ann(W^{\btd d}))] \in \mathcal{H}ilb^{sm}_{8}(\AA^{n}).
\]
This contradicts Condition (b).
\end{proof}

Finally we present the proof of the  characterization of points of the second irreducible component of the Grassmann cactus variety.

\begin{proof}[Proof of Theorem \ref{t:segre-veronese_map_143_general_case} and Corollary \ref{c:segre-veronese_map_143}]
We first prove Part (ii) of Theorem \ref{t:segre-veronese_map_143_general_case} for $\eta_{8,3}(\nu_d(\PP^n))$ defined as in Equation~\eqref{eq:eta_is_projection_143}.

Let $\psi \colon \PP T_1 \times \Gr(3, T_2)  \to \Gr(3, T_{d})$ be given by $([y_0], [U]) \mapsto [y_0^{d-2}U]$ and let $q\colon (T_1\setminus \{0\}) \to \PP T_1$ be the natural map. Let
\begin{multline*}
C = \{(y_0, [U]) \in T_1\times \Gr(3, T_2) \mid \text{there exists a completion of } y_0 \text{ to a basis } (y_0,y_1,\ldots, y_n) \text{ of } T_1 \text{ such that } \\
\Apolar((U|_{y_0=1})^{\btd d}) \text{ has Hilbert function } (1,4,3)\}.
\end{multline*}

Note that the set from the statement is $\psi((q\times \operatorname{Id}_{\Gr(3, T_2)})(C))$. 
We define
\begin{equation*}
D = \{(y_0, [U]) \in C \mid  [\Spec \Apolar ((U|_{y_0=1})^{\btd d})]\notin \Hilb_{8}^{sm}(\AA^n)\}.
\end{equation*}

We claim that the set $C$ is irreducible, $D$ is dense in $C$, and that $\dim D = \dim C = 8n + 9$. Consider the morphism $\varphi : GL(T_1)\times \Gr(3,T_2) \to \Gr(3,T_2)$, given by a change of basis. We have a product morphism 
\begin{equation*}
  \tau : GL(T_1)\times \Gr(3,T_2) \to T_1 \times\Gr(3,T_2)\text{, given by } (a, [U]) \mapsto (a(x_0), \varphi(a,[U]))\text{.}
\end{equation*}
Recall the sets $A, B$ from Lemma \ref{l:D_i_is_dense_and_of_dimension_7n+8}. 
Let $\chi:S_{\leq 2} \to T_2$ be the inverse of the $\CC$-linear isomorphism $T_2 \to S_{\leq 2}$ given by $P \mapsto (P|_{x_0=1})^{\btd d}$.
We have $\tau(GL(T_1)\times {A}) = C$ and $\tau(GL(T_1)\times {B}) = D$.
These follow from the definitions of the sets $A,B,C,D$ and the identity
  \[
  (\varphi(a, \chi(W))|_{a(x_0)=1})^{\btd d} = \langle f(a(x_1), \ldots, a(x_n)) \mid f\in W\rangle \text{ for every } [W]\in \operatorname{Gr}(3,S_{\leq 2}) \text{ and } a\in GL(T_1).
  \]

It follows from Lemma \ref{l:D_i_is_dense_and_of_dimension_7n+8} that $C$ is irreducible, $D$ is dense in $C$, and that $\dim D = \dim C = 8n +9$.
By Proposition~\ref{p:F3_tresciwy_to_poza_eta_143} if $(y_0, [U])\in D$ and $V=y_0^{d-2}U$ then $[V]\in \eta_{8,3} ( \nu_{d} (\PP^n)) $. Hence we have
the inclusion $\overline{\psi((q\times \operatorname{Id}_{\Gr(3, T_2)})(C))} \subseteq \eta_{8,3}(\nu_{d} (\PP^n))$.

Now we prove that, in fact, $\overline{\psi((q\times \operatorname{Id}_{\Gr(3, T_2)})(C))} = \eta_{8,3}(\nu_{d} (\PP^n))$.
 It follows from Proposition \ref{p:dim_eta_143}  that for every $ d  \geq 5$ we have
 \[\dim(\eta_{8,3}(\nu_{d} (\PP^n))) \leq 8n+8 \leq \dim ((q\times \operatorname{Id}_{\Gr(3, T_2)})(C)) = \dim (\overline{(q\times
 \operatorname{Id}_{\Gr(3, T_2)})({C})})   = \dim \overline{\psi((q\times \operatorname{Id}_{\Gr(3, T_2)})(C))}.\]
 The last equality follows from  \cite[Thm. 11.4.1]{RV2017}, since the fibers of $\psi$ are finite.
 Hence $\overline{\psi((q\times \operatorname{Id}_{\Gr(3, T_2)})(C))} = \eta_{8,3}(\nu_{d} (\PP^n))$.
 This concludes the proof of Theorem \ref{t:segre-veronese_map_143_general_case}(ii).
 
 We proceed to the proof of Theorem \ref{t:segre-veronese_map_143_general_case}(i). By Theorem \ref{t:143_summary}, the scheme $\mathcal{H}ilb_{8}(\PP^n)$ has two irreducible components. Thus, the variety $\kappa_{8,3}(\nu_d(\PP^n))$ has at most two irreducible components: $\sigma_{8,3}(\nu_d(\PP^n))$ and $\eta_{8,3}(\nu_d(\PP^n))$ by Equations \eqref{eq:kappa_is_projection_143}, \eqref{eq:sigma_is_projection_143}, \eqref{eq:eta_is_projection_143}. Let $[W]\in B$, where $B$ is as in Lemma \ref{l:D_i_is_dense_and_of_dimension_7n+8}. Then by Proposition~\ref{p:F3_tresciwy_to_poza_eta_143} we get $[W^{hom, d-2}]\in \kappa_{8,3}(\nu_d(\PP^n)) \setminus \sigma_{8,3}(\nu_d(\PP^n))$. It is enough to show that $\eta_{8,3}(\nu_d(\PP^n)) \neq \kappa_{8,3}(\nu_d(\PP^n))$. By Part (ii) every $[V]\in \eta_{8,3}(\nu_d(\PP^n))$ is divisible by the $(d-2)$-nd power of a linear form. Therefore, $[\langle x_0^d, x_1^d, x_2^d \rangle]\in \kappa_{8,3}(\nu_d(\PP^n))\setminus \eta_{8,3}(\nu_d(\PP^n))$. 

Now we prove Corollary \ref{c:segre-veronese_map_143}.  
The Grassmann cactus variety $\kappa_{8,3}(\nu_d (\PP^4))$ has two irreducible components by Part (i) of Theorem \ref{t:segre-veronese_map_143_general_case}. Assume that $n=4$. 
Then in the above notation the closure of $(q\times \operatorname{Id}_{\Gr(3, T_2)})(C)$ in $\PP T_1 \times \Gr(3, T_2)$ has the maximal dimension
$8\cdot 4+8=40$. Thus $\overline{(q\times \operatorname{Id}_{\Gr(3, T_2)})(C)} = \PP T_1\times \Gr(3, T_2)$. It follows that $\eta_{8,3}(\nu_d(\PP^4))
= \overline{\psi((q\times \operatorname{Id}_{\Gr(3, T_2)})(C))} = \psi(\PP T_1 \times \Gr(3, T_2))$.
\end{proof}

In order to perform the last step of the algorithm in Theorem \ref{t:algorithm_143} we need to know the dimension of the tangent space to $\mathcal{H}_{143}$ at a generic point.

\begin{lemma}\label{l:dim_H143}
Let $n\geq 4$ and $[R]\in \mathcal{H}_{143}\subseteq \Hilb_{8}(\PP^n)$ be a non-smoothable subscheme. 
Then the dimension of the tangent space $\dim_\CC T_{[R]}\Hilb_{8} (\PP^n)$ equals $8n-7$. 
\end{lemma}
\begin{proof}
Let $R' \subseteq \PP^4$ be a subscheme abstractly isomorphic with $R$.  From \cite[Lem. 2.3]{CN09} we have 
\[
 \dim_\CC T_{[R]}\Hilb_{8} (\PP^n) = 8n + T_{[R']}\Hilb_{8} (\PP^4) -32.
\]
From \cite[Thm. 1.1]{BJ17} $R'$ is non-smoothable, hence $\dim T_{[R']}\Hilb_{8} (\PP^4) = 25$  by
\cite[Thm. 1.3 and the comment above]{CEVV09}.
\end{proof}

Using the description of the  irreducible component $\eta$ given in Theorem \ref{t:segre-veronese_map_143_general_case}, we are able to determine algorithmically if a given point from the Grassmmann cactus variety is in the Grassmann secant variety.

\begin{theorem}\label{t:algorithm_143}
Let $n$ be at least 4 and $T=\mathbb{C}[x_0,\ldots ,x_n]$ be a polynomial ring. Given an integer $d\geq 5$ and $[V] \in \kappa_{8,3}(\nu_{d}(\mathbb{P}T_1)) \subseteq \Gr(3,T_{d})$ the following algorithm checks if $[V]\in \sigma_{8,3}(\nu_{d}(\mathbb{P}T_1))$.
\begin{adjustwidth}{15pt}{0pt}
\begin{itemize}
\item[\textbf{Step 1}] Compute the ideal $\mathfrak{a} = \sqrt{((\Ann V)_{\leq d-2})}$.

\item[\textbf{Step 2}] If $\mathfrak{a}_1$ is not $n$-dimensional, then $[V]\in \sigma_{8,3}(\nu_{d}(\mathbb{P}T_1))$ and the algorithm terminates. 
Otherwise compute $ \{K\in T_1 \mid \mathfrak{a}_1 \lrcorner K = 0\}$. Let $y_0$ be a generator of this one dimensional $\mathbb{C}$-vector space.

\item[\textbf{Step 3}] Let $e$ be the maximal integer such that $y_0^e$ divides $V$. If $e\neq d-2$, then $[V] \in \sigma_{8,3}(\nu_{d}(\mathbb{P}T_1))$ and the algorithm terminates. Otherwise let $V=y_0^{d-2}U$, pick a basis $(y_0,y_1,\dots,y_n)$ of $T_1$ and compute $W=U|_{y_0=1}\subseteq R:=\mathbb{C}[y_1,\dots,y_n]$.

\item[\textbf{Step 4}]  Let  $I=\Ann (W^{\btd d}) \subseteq R^*$. If
  the Hilbert function of $R^*/I$ is not $(1,4,3)$, then $[V] \in \sigma_{8,3}(\nu_{d}(\PP T_1))$, and the algorithm terminates.
\item[\textbf{Step 5}]  Compute $r=\dim_{\mathbb{C}}\operatorname{Hom}_{R^*}(I, R^*/I)$. Then $[V] \in \sigma_{8,3}(\nu_{d}(\mathbb{P}T_1))$ if and only if $r>8n-7$. 
\end{itemize}
\end{adjustwidth}
\end{theorem}

The following lemma gives a description of the set-theoretic difference of the Grassmann cactus variety and the Grassmann secant variety. We need it to give a clear proof of Theorem \ref{t:algorithm_143}.

\begin{lemma}\label{l:algorithm_form_143}
Let $d\geq 5, n\geq 4$. The point $[V] \in \kappa_{8,3}(\nu_{d}(\PP^n))$
does not belong to $\sigma_{8,3}(\nu_{d}(\PP^n))$ if and only if there exists a
linear form $y_0 \in T_1$, and $U \in \Gr(3, T_2)$ such that $V = y_0^{d-2}U$ and for any
completion of $y_0$ to a basis of $T_1$ we have:
  \begin{enumerate}[label=(\alph*)]
    \item $\Apolar((U|_{y_0=1})^{\btd d})$ has Hilbert function $(1,4,3)$,
    \item $[\Spec \Apolar((U|_{y_0=1})^{\btd d})] \notin
      \mathcal{H}ilb^{sm}_{8}(\AA^n)$.
   \end{enumerate}
\end{lemma}
\begin{proof}
If $y_0 \in T_1$ and $U \in \Gr(3, T_2)$ are such that $V = y_0^{d-2}U$, and there exists
  a completion of $y_0$ to a basis $(y_0,\dots,y_n)$ of $T_1$, for which
  Conditions (a),(b) hold, we get 
  \begin{equation*}
    [V] \notin \sigma_{8,3}(\nu_{d}(\PP^n))
  \end{equation*}
  by Proposition \ref{p:F3_tresciwy_to_poza_eta_143}.

  Assume that $[V] \notin \sigma_{8,3}(\nu_{d}(\PP^n))$. Then by Theorem
  \ref{t:segre-veronese_map_143_general_case} there exists a linear form $y_0 \in T_1$ such that
  $y_0^{d-2} | V$. Using Proposition \ref{p:l5l_of_cactus_rank_at_most_4_143} we
  conclude that $V$ is not divisible by $y_0^{d-1}$. Hence we showed that $V = y_0^{d-2} U$
  for some $U \in \Gr(3, T_2)$. Extend $y_0$ to a basis
  $(y_0,y_1,\dots,y_n)$. Let $W = U|_{y_0 = 1}$.

  Now we prove Conditions  (a), (b) hold. We have 
\[
  V=(W^{\btd d})^{hom, d-2}.
\]
  By Lemma \ref{l:subspace_case_annihilators} (i)
  \begin{equation*}
    \Ann(W^{\btd d})^{hom} \subseteq \Ann(V)\text{.}
  \end{equation*}
  If  $\dim_\CC (\Apolar(W^{\btd d})) \leq 7$, then $\crr(V) \leq 7$ by the Cactus Apolarity Lemma \ref{p:cactus_apolarity}, since $\Ann(W^{\btd d})^{hom}$ is saturated by Lemma \ref{l:homogenization_is_saturated}. Therefore, $[V]\in \kappa_{7,3}(\nu_d(\PP^n)) = \sigma_{7,3}(\nu_d(\PP^n)) \subseteq \sigma_{8,3}(\nu_d(\PP^n))$, a contradiction.
  
  From Theorem \ref{t:general_subspace}(ii) we obtain $\dim_\CC (\Apolar(W^{\btd d})) \leq 8$. We proved that $\dim_\CC (\Apolar(W^{\btd d})) = 8$. Since we assumed that $[V] \notin \sigma_{8,3}(\nu_d(\PP^n))$, it follows by Lemma \ref{l:smoothable_then_br_of_hom_at_most_r}  that
  $\Spec(\Apolar(W^{\btd d}))$ is not smoothable. This implies Condition (b) holds.
  From \cite[Thm. 4.20]{CEVV09}, the algebra $\Apolar(W^{\btd d})$ has Hilbert function $(1,4,3)$. We proved Condition (a) holds.
\end{proof}

Steps 2--5 of the algorithm check whether $V$ is of the form given by Lemma \ref{l:algorithm_form_143}.

\begin{proof}[Proof of Theorem \ref{t:algorithm_143}]
Assume that $[V]\notin \sigma_{8,3}(\nu_{d}(\PP^n))$. Then there exist a basis
$(y_0,\ldots, y_n)$ of $T_1$ and $U\subseteq \CC[y_0,\ldots, y_n]$ as in Lemma
\ref{l:algorithm_form_143}. Let $W = U |_{y_0=1} \subseteq \CC[y_1,y_2,...,y_n]$.
Recall that
\begin{align*}
  W^{\btd d} := \{ (d-2)! F_2 + (d-1)! F_1 + d! F_0, & \text{ where } F_2 + F_1 + F_0 \in W, \\ & \text{ and } F_i \in \CC[y_1,y_2, \ldots , y_n]_i \}\text{.}
\end{align*}

Then, in the notation from Definition \ref{d:homogenization_of_space}, we get 
\[
  V= (W^{\btd d})^{hom,d-2}.
\]
By Lemma
\ref{l:subspace_case_annihilators}(ii), we have $\Ann (V)_{\leq d-2} =
(\Ann(W^{\btd d})^{hom})_{\leq d-2}$. Moreover, since $W^{\btd d} \subseteq \CC[y_1,\dots,y_n]_{\leq 2}$, and $d \geq 5$, we obtain
$((\Ann(W^{\btd d})^{hom})_{\leq d-2}) = \Ann(W^{\btd d})^{hom}$. Therefore we have
\[
  \mathfrak{a} = \sqrt{(\Ann(V )_{\leq d-2})} = \sqrt{\Ann(W^{\btd d})^{hom}} = (\beta_1,\ldots, \beta_n),
\]
where $\beta_1,\ldots, \beta_n \in T^*_1$ are dual to $y_1,\ldots, y_n \in T_1$. This shows that if the $\CC$-linear space $\big(\sqrt{(\Ann(V)_{\leq d-2})}\big)_1$ is not $n$-dimensional, then $[V] \in \sigma_{8,3}(\nu_{d}(\PP^n))$. Therefore, in that case, the algorithm stops correctly at Step~2.

Assume that the algorithm did not stop at Step 2. Then if $V$ is of the form as in Lemma \ref{l:algorithm_form_143}, then $y_0$ divides $V$ exactly $(d-2)$-times. Otherwise $[V]\in \sigma_{8,3}(\nu_{d}(\PP^n))$ and the algorithm stops correctly at Step 3.

Assume that the algorithm did not stop at Step 3. Then the Hilbert function of $R^*/I$ computed in Step 4
is $(1,4,3)$ if and only if Condition (a) of Lemma \ref{l:algorithm_form_143} is
fulfilled. Therefore, if it is not $(1,4,3)$, the algorithm stops correctly at Step 4. 

Assume that the algorithm did not stop at Step 4. Then $V$ satisfies Condition
(a) from Lemma \ref{l:algorithm_form_143}. Hence $[V]$ is in
$\sigma_{8,3}(\nu_{d}(\PP^n))$ if and only if $V$ does not satisfy Condition
(b). Using Lemma \ref{l:dim_H143}, this is equivalent to
\[
  \dim_{\mathbb{C}}\operatorname{Hom}_{R^*}(I, R^*/I) > 8n -7\text{.}
\]
The left term is the dimension of the tangent space to the Hilbert scheme $\mathcal{H}ilb_{8}(\AA^n)$ at the point $[\Spec R^*/I]$ (see \cite[Prop. 2.3.]{H09} or \cite[Thm. 18.29]{MS04}).
\end{proof}

\appendix
\section{Construction of the morphism to the Hilbert scheme}\label{s:hilbert}
Let $\kk$ be an algebraically closed field, $S^* = \kk[\alpha_1,\alpha_2,...,\alpha_n]$ be a polynomial ring and consider its graded dual $S=\kk_{dp}[x_1,x_2,...,x_n]$. In this section we prove the following theorem, which is used in Sections \ref{s:14thsecant} and \ref{s:83grassmann}.

\begin{theorem}\label{t:morphism}
  Let $l,m,r$ be positive integers. Consider a locally closed reduced subscheme $E$ of $\Gr(l, S_{\leq m})$ whose closed points satisfy
    \[
      E(\Bbbk) \subseteq \{[W] \in \Gr(l,S_{\leq m})\mid \Spec {S}^*/\Ann(W) \text{ has length } r\}\text{.}
    \]
  The natural map from $E$ to the Hilbert scheme of $r$ points in $\AA^n$, given on closed points by 
  $[W] \mapsto [\Spec {S}^*/\Ann(W)]$,
  is a morphism of $\Bbbk$-schemes.
\end{theorem}

For a $\kk$-algebra $A$ we denote by $S_A$ and by $S^*_A$ the $A$-algebras $S \otimes_{\kk} A$ and $S^* \otimes_{\kk} A$, respectively.
Given a $\kk$-algebra homomorphism $\varphi: A \to B$ we will denote by the same letter the induced homomorphisms $S_A \to S_B$ and $S^*_A \to S^*_B$. 
For any $A$-submodule $W$ of $S_A$, by $S^*_A \lrcorner W$ we denote the $S^*_A$-submodule of $S_A$ generated by $W$. 
Given $t$ in $\Spec A$ we denote by $k(t)$ the residue field of $t$ on $\Spec A$ and we denote by $\iota_t$ the natural morphism from $A$ to $k(t)$.

\begin{lemma}\label{l:base_change_of_annihilators}
Let $\varphi\colon A\to B$ be a morphism of $\Bbbk$-algebras and $W\subseteq S_A$. Then the natural map
\[
(S_A^* \lrcorner W)\otimes_A B \to S_B
\]
surjects onto $S_B^*\lrcorner \varphi(W)$.
\end{lemma}
\begin{proof}
Let $\theta\in S_A^*$, $f\in W$ and $b\in B$. Then $(\theta \lrcorner f)\otimes_A b \mapsto b(\varphi(\theta)\lrcorner \varphi(f))$, so the image of $(S_A^* \lrcorner W)\otimes_A B$ is contained in $S_B^*\lrcorner \varphi(W)$.
Let $\eta \lrcorner \varphi(f) \in S_B^*\lrcorner \varphi(W)$ with $\eta = \sum_\mathbf{u} b_\mathbf{u}\alpha^\mathbf{u}$ for some $f\in W$ and $b_\mathbf{u}\in B$. Then
\[
\sum_\mathbf{u} (\alpha^\mathbf{u} \lrcorner f) \otimes b_\mathbf{u} \mapsto (\eta \lrcorner \varphi(f)).
\]
\end{proof}

In some special cases, the surjection from Lemma \ref{l:base_change_of_annihilators} is in fact an isomorphism.
\begin{corollary}\label{c:base_change_when_flat}
If $S_A/(S_A^*\lrcorner W)$ is a flat $A$-module or if $B$ is a flat $A$-module (for instance if $B=A_\mathfrak{p}$), then the map
\[
(S_A^* \lrcorner W)\otimes_A B \to S_B^*\lrcorner \varphi(W)
\]
from Lemma \ref{l:base_change_of_annihilators}
is an isomorphism.
\end{corollary}
\begin{proof}
  By Lemma \ref{l:base_change_of_annihilators} it is enough to show that the natural map 
\[
(S_A^* \lrcorner W)\otimes_A B \to S_B
\]
is injective.
This follows from the $\operatorname{Tor}$ exact sequence given by application of the functor $- \otimes_A B$ to the short exact sequence
  
  \[
   0 \to (S_A^*\lrcorner W) \to S_A \to S_A/(S_A^*\lrcorner W) \to 0.
  \]

\end{proof}

\begin{lemma}\label{l:annihilators_and_homs}
 Let $A$ be a $\kk$-algebra and $W$ be a finite $A$-submodule of $S_A$. 
 Then
 $\Hom_A(S^*_A \lrcorner W,A) \simeq S_A^*/\Ann(W)$.
\end{lemma}
\begin{proof}
Let $N=S^*_A \lrcorner W$ and define a homomorphism $\psi_A: S_A^* \to \Hom_A(N,A)$, by $(\psi_A (\theta))(f) = (\theta \lrcorner f)_0 $.
 We have a factorization of $\psi_A$ by $S_A^*/\Ann(W)$.
 
We shall show that $\ker(\psi_A)\subseteq \Ann (W)$. Let $\theta \in \ker(\psi_A)$ and $f\in W$.
Let $\theta \lrcorner f = G_d + \ldots + G_0$ for a positive integer $d$. Then for every $j\in \{0,\ldots, d\}$
and $\eta \in (S_A^*)_j$ we have $0 =  (\theta \lrcorner (\eta \lrcorner f))_0 = \eta \lrcorner G_j$. Thus $G_j = 0$ and hence $\theta \in \Ann(f)$. Since $f\in W$ was arbitrary, we have $\theta \in \Ann(W)$.

We proceed to showing that $\psi$ is surjective. We first assume that $(A, \mathfrak{m})$ is a local ring. Let $\varphi \in \Hom_A(N,A)$ and assume that $g_1,\ldots, g_s$ is a minimal set of generators of the $A$-module $N=S_A^*\lrcorner W$. Let $\mathcal{M}$ be the set of divided power monomials in $S_A$ of degree at most $n_0 =\max \{\deg(f) \mid f\in W\}$. Form a matrix $M$ over $A$ with rows corresponding to $g_1,\ldots, g_s$ and entries equal to coordinates of $g_i$ in the basis $\mathcal{M}$. Then there exists an invertible $s\times s$ minor of $M$. 
Indeed, otherwise all minors are in the maximal ideal of $A$ and therefore $\overline{g_1}, \ldots, \overline{g_s} \in N/\mathfrak{m}N$ are $A/\mathfrak{m}$-linearly dependent. Thus, by Nakayama's lemma, $g_1,\ldots, g_s$ is not a minimal set of generators.

Let $a_i = \varphi(g_i)$ for $i=1,\ldots, s$. 
If we write $\theta \in S_A^*$ as a vector $\mathbf{v}$ in the basis dual to $\mathcal{M}$, then $\psi_A(\theta)(g_i)$ is the $i$-th coordinate of the vector $M\cdot \mathbf{v}$. Therefore, there exists $\theta\in S_A^*$ with $\psi_A(\theta)=\varphi$, as long as there exists $\mathbf{v}$ with $M\cdot \mathbf{v} = [a_1,\ldots, a_s]^T$. 
Therefore, it is enough to show that $M$ gives a surjective morphism $A^{\# \mathcal{M}} \to A^s$. Let $M'$ be a $s\times s$ submatrix of $M$ with invertible determinant. We will show that $M'$ defines a surjective morphism $A^s \to A^s$.
Let $M'^D$ be the adjoint matrix of $M'$. Given $\mathbf{w} \in A^s$, we have
$\mathbf{w}=M'\cdot \mathbf{v} $ for $\mathbf{v} = \frac{1}{\det M'} M'^D \cdot \mathbf{w}$.

Let $A$ be an arbitrary $\Bbbk$ -algebra and $Q$ be the cokernel of $\psi_A$. We claim that $Q=0$. It is enough to show that $Q_\mathfrak{p} = 0$ for all $\mathfrak{p}\in \Spec A$. 
Let $l : A \to A_\mathfrak{p}$ be the localization. Then $N_\mathfrak{p} \simeq  S^*_{A_\mathfrak{p}} \lrcorner l(W) \simeq S^*_{A_\mathfrak{p}} \lrcorner W_\mathfrak{p}$ (the first isomorphism follows from Corollary \ref{c:base_change_when_flat}).
 Therefore, by the local case considered before, it is enough to show that $(\psi_A)_\mathfrak{p} = \psi_{A_\mathfrak{p}}$.

Using isomorphisms $(S^*_A)_\mathfrak{p} \simeq S^*_{A_\mathfrak{p}}$, $S_{A_\mathfrak{p}}^* \lrcorner W_\mathfrak{p} \simeq N_\mathfrak{p}$  and $(\Hom_A(N, A))_\mathfrak{p} \simeq \Hom_{A_\mathfrak{p}}(N_\mathfrak{p}, A_\mathfrak{p})$ we can write for $\theta \in S_A^*, f\in N$, $a,b\in A\setminus\mathfrak{p}$:
\[
\Bigg(\psi_{A_\mathfrak{p}}\Big(\frac{\theta}{a}\Big)\Bigg)\Big(\frac{f}{b}\Big) = \Big(\frac{\theta}{a} \lrcorner \frac{f}{b}\Big)_0
\]
and
\[
\Bigg((\psi_A)_\mathfrak{p}\Big(\frac{\theta}{a}\Big)\Bigg) \Big(\frac{f}{b}\Big) = \frac{\psi_A(\theta)}{a} \Big(\frac{f}{b}\Big) = \frac{\big(\psi_A(\theta)\big)(f)}{ab} = \frac{(\theta \lrcorner f)_0}{ab}.
\]
\end{proof}

The following lemma is a slight modification of \cite[Prop. 2.12]{Jel16}.
Recall for $t$ a point of $\Spec A$, we denote by $\iota_t$ the natural map $S_A \to S_{k(t)}$.
\begin{lemma}\label{l:joachim}
 Let $l, m \in \mathbb{Z}_{\geq 1}$, let $A$ be a Noetherian $\kk$-algebra, and let
  $[W]$ be a $(\Spec A)$-point of $\Gr(l,S_{\leq m})$, i.e. $W$ is an
  $A$-submodule of $(S_A)_{\leq m}$ such that the quotient module is locally
  free of rank $\dim_\kk S_{\leq m} - l $.
 Define $a_W : \Spec ({S_A}^*/\Ann(W)) \to \Spec A$ to be the natural map. Then the following holds:
 \begin{enumerate}[label=(\roman*)]
  \item If $A$ is a reduced finitely generated $\kk$-algebra and the length of $ S_{k(t)}^*/\Ann(\iota_t(W))$ is independent of the choice of a closed point $t \in \Spec A$ then $S_A/(S^*_A \lrcorner W)$ and $ S_A^*/\Ann(W)$ are flat $A$-modules. 
  \item If $W$ is such that $S_A/(S^*_A \lrcorner W)$ is a flat $A$-module,
  then the base change of $a_W$ via any homomorphism between Noetherian rings $\varphi : A \to B$ is equal to 
  \[\Spec (S_B^*/ \Ann(\varphi(W))) \to \Spec B\]
  In particular, the fiber of $a_W$ over $t \in \Spec A$ is naturally $\Spec S_{k(t)}^*/\Ann(\iota_t(W))$.
  \end{enumerate}
\end{lemma}
\begin{proof}
\hfill
\begin{enumerate}[label=(\roman*)] 
 \item First we prove that $S_A/(S^*_A \lrcorner W)$ is a flat $A$-module. We know that
   \begin{equation*}
     S_A \cong (S_A)_{\leq m} \oplus (S_A)_{> m}\text{.}
   \end{equation*}
   Since $(S_A)_{> m}$ is a free $A$-module, it suffices to show that $(S_A)_{\leq m}/(S^*_A\lrcorner W)$ is $A$-flat. Denote this module by $P$.

   This module is finitely generated, hence $P$ is flat if and only if $P$ is locally free. Now $A$ is reduced and finitely generated, so  $P$ is $A$-flat if and only if it has locally constant rank: $\dim_{k(t)}(P \otimes k(t))$ is independent of the choice of a closed point $t \in \Spec A$. 

  We have an exact sequence
  \begin{equation*}
    0 \to S^*_A \lrcorner W \to (S_A)_{\leq m} \to P \to 0\text{.}
  \end{equation*}
  We tensor it by $k(t)$, getting the exact sequence
\begin{equation*}
  (S^* \lrcorner W)\otimes_A k(t) \xrightarrow{u} (S_A)_{\leq m} \otimes_A k(t) \to P \otimes_A k(t) \to 0\text{.}
\end{equation*}
Then
\begin{align*}
  \dim_{k(t)} (P \otimes_A k(t)) & \hspace{0.75cm} = \hspace{0.73cm}\dim_{k(t)} ((S_A)_{\leq m} \otimes_A k(t)) - \dim_{k(t)} \image u \\
  & \stackrel{\text{by Lemma \ref{l:base_change_of_annihilators}}}{=} \dim_{\kk} S_{\leq m} - \dim_{k(t)} S^*_{k(t)} \lrcorner \iota_t(W) \\
  & \stackrel{\text{by Lemma \ref{l:annihilators_and_homs}}}{=} \dim_{\kk} S_{\leq m} - \dim_{k(t)} S^*_{k(t)} /\Ann(\iota_t(W))\text{,}
\end{align*}
which is constant by assumption.

It remains to prove that  $ S^*/\Ann(W)$ is a flat $A$-module. It follows from Lemma \ref{l:annihilators_and_homs}, that 
\[ 
 S^*/\Ann(W) \simeq \Hom_A (S^*_A \lrcorner W,A). 
\]
Since $S^*_A \lrcorner W$ is the kernel of a surjection of flat $A$-modules,  it is a flat $A$-module. Because it is finite as an $A$-module, it is locally free of finite rank. Therefore $\Hom_A (S^*_A \lrcorner W,A)$ is a locally free $A$-module of finite rank, thus flat.
\item
 Let $N = S^*_A \lrcorner W$. Suppose that  $S_A / N$ is a flat $A$-module. By Corollary \ref{c:base_change_when_flat} the natural morphism $N\otimes_A B \to S_A\otimes_A B \simeq S_B$ sends $N\otimes_A B$ isomorphically to $S^*_B \lrcorner \varphi(W)$. By Lemma \ref{l:annihilators_and_homs} $S^*_A/ \Ann(W) \simeq \Hom_A (S^*_A \lrcorner W, A)$ and $S^*_B/ \Ann(\varphi(W)) \simeq \Hom_B (S^*_B \lrcorner \varphi(W), B)$.
 Thus it is enough to show that $\Hom_A (N, A) \otimes_A B \simeq \Hom_B ({N\otimes_A B}, B)$. This follows from \cite[Exe.~7.20(a)]{GW10} and the fact that $N = S^*_A \lrcorner W$ is flat and finitely generated over a Noetherian ring, hence locally free of finite rank, see \cite[Prop.~4.4.3]{Bos12}.
\end{enumerate}
\end{proof}

\begin{lemma}\label{l:two_definitions_agree}
  Let $W \subseteq (S_A)_{\leq m}$ be an $A$-submodule, and let $Q = (S_A)_{\leq m}/ W$. Let $t \in \Spec A$ be any closed point. If $Q$ is $A$-flat, then $\iota_t(W) = W\otimes_A k(t)$.
\end{lemma}
\begin{proof}
  Consider the following commutative diagram:
  \[\begin{tikzcd}
      0 \ar{r} & W \ar{r}\ar{d}{a} & (S_A)_{\leq m} \ar{r}\ar{d} & Q\ar{d} \ar{r} & 0 \\
      0 \ar{r} & W\otimes_A k(t)\ar{r}{b} & (S_{k(t)})_{\leq m} \ar{r} & Q\otimes_A k(t) \ar{r} & 0 \text{.}
  \end{tikzcd}\]
  The map $a$ is a surjection since tensoring is right-exact. The map $b$ is an injection, because $Q$ is $A$-flat. Hence $\iota_t(W) = W\otimes_A k(t)$.
\end{proof}

\begin{proof}[Proof of Theorem \ref{t:morphism}]
  Take any cover of $E$ by open affines $\Spec A_i$. We construct morphisms
  \begin{equation*}
    \varphi_i : \Spec A_i \to \Hilb_r(\mathbb{A}^n)\text{,}
  \end{equation*}
  and finally we show that these morphisms glue.

  Let $\mathcal{U}$ be the universal subbundle on $\Gr(l, S_{\leq m})$, treated as a locally free sheaf. Let $\mathcal{U}|_{\Spec A_i} = \widetilde{W_i}$, where $W_i \subseteq S_{\leq m}\otimes A_i$ is a submodule. Observe that $(S_{A_i})_{\leq m}/W_i$ is $A_i$-flat from the definition of the Grassmann functor, see \cite[\S8.4]{GW10}.

  Our morphism will be defined by the family $\Spec S^*_{A_i}/\Ann(W_i) \to
  \Spec A_i$. We know that the scheme $\Spec S^*_{A_i}/\Ann(W_i)$ is a closed subscheme of
  $\mathbb{A}^{n}_{A_i}$. We want to use Part (i) of Lemma \ref{l:joachim}
  for $W_i$. In order to do it, it suffices to show that for every closed point $t \in \Spec A_i$,
  the vector space $S^*_{k(t)}/\Ann(\iota_t(W_i))$ has dimension $r$. But this follows
  from the fact that $[W_i\otimes k(t)] =[\iota_t(W_i)]$ by Lemma
  \ref{l:two_definitions_agree}, and the fact that $[W_i \otimes k(t)] \in E \subseteq
  \Gr(l, S_{\leq m})$.  Hence, both modules $S^*_{A_i}/\Ann(W_i)$ and
  $S_{A_i}/S^*_{A_i}\lrcorner W_i$ are $A_i$-flat. Then we can use Part (ii) of Lemma \ref{l:joachim} to show that our
  family has fibers of length $r$. Hence, from the defining property of the
  Hilbert scheme, we have a morphism $\varphi_i : \Spec A_i \to
  \Hilb_r(\mathbb{A}^n)$. Moreover, the fiber of the family $\Spec
  S^*_{A_i}/\Ann(W_i)$ over the closed point $t$ is $S^*_{k(t)}/\Ann(\iota_t(W_i))$.
  Therefore $\varphi_i$ on closed points is defined by
  \begin{equation*}
    [W] \mapsto [\Spec S^*/\Ann(W)]\text{.}
  \end{equation*}
  Since the morphims $\varphi_i$ are defined on closed points by the same formula, they glue together.
\end{proof}

\section{Implementation of the algorithm in Macaulay2}\label{appendix}
We present the code of the algorithm from Theorem \ref{t:algorithm} (for $n=6$) written in Macaulay2 \cite{M2}.

\begin{verbatim}
KK=ZZ/7919
T=KK[x_0..x_6]

completeToBasis = (y) -> {
  use T;
  L := {y,x_0,x_1,x_2,x_3,x_4,x_5,x_6};
  A := {x_0,x_1,x_2,x_3,x_4,x_5,x_6};
  for i from 1 to #L-1 do{
    (M,C) := coefficients(matrix{drop(L, {i,i})},Monomials=> A);
    if rank(C) == 7 then return drop(L, {i,i});
  }
}

triangle = (d,f) -> {
  C := terms(f);
  C = apply(C, g -> (d-(degree g)#0)! * g);
  return sum(C);
}

generatorsUpToDegree = (d,I) -> {
  E := entries mingens I;
  E = E#0;
  E = select(E, (i)->((degree i)#0 <= d));
  return ideal E;
}

annihilatorUpToDegree = (d,G) -> {
  J := inverseSystem(G);
  return generatorsUpToDegree(d, J);
}

dualLinearGenerator = (I) -> {
  J := generatorsUpToDegree(1,I);
  K := inverseSystem(J);
  J = generatorsUpToDegree(1,K);
  y := entries mingens J;
  y = y#0;
  return y#0;
}

howManyTimes = (y,G) -> {
  i := 0;
  while (G % y) == 0 do{
    G=G//y;
    i=i+1;
  };
  return i;
}

dehomogenizationWrtBasis = (G, L) -> {
  y := L#0;
  R := T/ideal(y-1);
  G = substitute(G, R);
  Q := KK[L#1, L#2, L#3, L#4, L#5, L#6];
  q := map(R, Q, {L#1,L#2,L#3,L#4,L#5,L#6});
  J := preimage_q(ideal(G));
  E := entries mingens J;
  E = E#0;
  return (E#0, Q);
}

homogeneousPart = (d, G) -> {
  E := terms G;
  E = select(E, (i)->((degree i)#0 == d));
  return sum E;
}

localHilbertFunction = I -> {
  S := ring I;
  m := ideal vars S;
  R := S/I;
  m = sub(m, R);
  return apply({ R/m, m/m^2, m^2/m^3, m^3/m^4}, degree);
}
  

isInSecant = (G) -> {
  --Step 1:
  d := (degree(G))#0 - 3;
  I := annihilatorUpToDegree(d,G);
  J := radical(I);
  --Step 2:
  if (hilbertFunction(1, module(J)) != 6) then return true;
  y := dualLinearGenerator(J);
  --Step 3:
  if (howManyTimes(y, G) != d) then return true;
  --Step 4:
  for i from 0 to d-1 do G=G//y;
  L := completeToBasis(y);
  (f, R) := dehomogenizationWrtBasis(G, L);
  ftriangle = triangle(d+3,f);
  K := inverseSystem(ftriangle);
  if (localHilbertFunction(K) != {1,6,6,1}) then return true;
  --Step 5:
  deg := degree Hom(K, R/K);
  return (deg > 76);
}
\end{verbatim}

\bibliographystyle{abbrv}

\end{document}